\documentclass[english]{amsart}
\usepackage[T1]{fontenc}
\usepackage[latin9]{inputenc}
\usepackage{geometry}
\usepackage{color}
\usepackage{babel}
\usepackage{mathrsfs}
\usepackage{amsthm}
\usepackage{amsbsy}
\usepackage{amstext}
\usepackage{amssymb}
\usepackage{setspace}
\usepackage[unicode=true,pdfusetitle,
 bookmarks=true,bookmarksnumbered=false,bookmarksopen=false,
 breaklinks=false,pdfborder={0 0 1},backref=page,colorlinks=true]
 {hyperref}
\hypersetup{
 colorlinks=true,citecolor=blue,linkcolor=blue,linktocpage=true}

\makeatletter
\numberwithin{equation}{section}
\numberwithin{figure}{section}
\theoremstyle{plain}
\newtheorem{thm}{\protect\theoremname}[section]
  \theoremstyle{remark}
  \newtheorem{rem}[thm]{\protect\remarkname}
  \theoremstyle{plain}
  \newtheorem{lem}[thm]{\protect\lemmaname}
  \theoremstyle{plain}
  \newtheorem{cor}[thm]{\protect\corollaryname}
  \theoremstyle{plain}
  \newtheorem{prop}[thm]{\protect\propositionname}
  \theoremstyle{definition}
  \newtheorem{example}[thm]{\protect\examplename}

\allowdisplaybreaks

\makeatother

  \providecommand{\corollaryname}{Corollary}
  \providecommand{\examplename}{Example}
  \providecommand{\lemmaname}{Lemma}
  \providecommand{\propositionname}{Proposition}
  \providecommand{\remarkname}{Remark}
\providecommand{\theoremname}{Theorem}

\begin{document}
\begin{singlespace}

\title[Intersections of Cantor Sets]{On Intersections of Cantor Sets: Self-Similarity}
\end{singlespace}

\author{Steen Pedersen and Jason D. Phillips}

\address{Department of Mathematics, Wright State University, Dayton OH 45435.}

\email{steen@math.wright.edu}

\email{phillips.50@wright.edu}


\begin{abstract}
Let \emph{C} be a Cantor set. For a real number \emph{t} let \emph{C+t}
be the translate of \emph{C} by \emph{t}, We say two real numbers
\emph{s},\emph{t} are equivalent if\emph{ }the intersection of \emph{C}
and \emph{C+s} is a translate of the intersection of \emph{C} and
\emph{C+t}. We consider a class of Cantor sets determined by similarities
with one fixed positive contraction ratio. For this class of Cantor
set, we show that an ``initial segment'' of the intersection of
\emph{C} and \emph{C+t} is a self-similar set with contraction ratios
that are powers of the contraction ratio used to describe \emph{C}
as a self-similar set if and only if \emph{t} is equivalent to a rational
number. Our results are new even for the middle thirds Cantor set. 
\end{abstract}
\maketitle
\tableofcontents{}

\section{\label{Sec-1:Introduction}Introduction}

Let $n\ge3$ be an integer. Any real number $t\in\left[0,1\right]$
has at least one $n$\emph{-ary} \emph{representation} 
\[
t=0._{n}t_{1}t_{2}\cdots=\sum_{k=1}^{\infty}\frac{t_{k}}{n^{k}}
\]
where each $t_{k}$ is one of the digits $0,1,\ldots,n-1.$ Deleting
some elements from the full digit set $\{0,1,\ldots n-1\}$ we get
a set of \emph{digits} $D:=\left\{ d_{k}\mid k=1,2,\ldots,m\right\} $
with $d_{k}<d_{k+1}$ for all $k=1,2,\ldots,m-1.$ Assuming $2\leq m<n$
we get a corresponding \emph{deleted digits Cantor set} 
\begin{equation}
C=C_{n,D}:=\left\{ \sum_{k=1}^{\infty}\frac{x_{k}}{n^{k}}\,\Big|\, x_{k}\in D\text{ for all }k\in\mathbb{N}\right\} .\label{Sec-1-eq:C-defined}
\end{equation}

We say that $D$ is \emph{uniform,} if $d_{k+1}-d_{k},$ $k=1,2,\ldots,m-1$
is constant and $\geq2.$ We say $D$ is \emph{regular,} if $D$ is
a subset of a uniform digit set. Finally, we say that $D$ is \emph{sparse,}
if $\left|\delta-\delta'\right|\geq2$ for all $\delta\neq\delta'$
in the set of differences 
\[
\Delta:=D-D=\left\{ d_{j}-d_{k}\mid d_{j},d_{k}\in D\right\} .
\]
Clearly, a uniform set is regular and a regular set is sparse. The
set $D=\left\{ 0,5,7\right\} $ is sparse and not regular. We will
abuse the terminology and say $C_{n,D}$ is uniform, regular, or sparse
provided $D$ has the corresponding property. The middle thirds Cantor
set is obtained by setting $n=3$ and $D=\{0,2\}.$ In particular,
the middle thirds Cantor set is a uniform set. 

In this paper we investigate self-similarity properties of the intersections
$C\cap\left(C+t\right)$ of $C$ with its translates $C+t:=\{x+t\mid x\in C\},$
for sparse Cantor sets $C$. Using a geometric approach, we investigate
the class of real numbers $t\in\left[0,1\right]$ for which the intersection
$C\cap\left(C+t\right)$ can be expressed as the finite, disjoint
union of self-similar sets. Since the problem is invariant under translation,
we will assume $d_{1}=0.$ 

Compared to previous studies, e.g., \cite{DHW08}, \cite{LYZ11},
\cite{KLD12}, and \cite{ZLL08}, of self-similarity properties of
$C\cap\left(C+t\right)$ we allow a greater class of digits sets,
sparse sets as compared to uniform sets and we study self-similarity
of a subset of $C\cap\left(C+t\right)$ instead of instead of self-similarity
of all of $C\cap\left(C+t\right)$.

\subsection{Statement of results}

Fix a real number $t$. If $C\cap\left(C+t\right)$ is non-empty,
let 
\[
C\left(t\right):=\left(C\cap\left(C+t\right)\right)-\inf\left(C\cap\left(C+t\right)\right),
\]
otherwise, let $C(t)$ be the empty set. We say that two real numbers
$s$ and $t$ are \emph{translation equivalent}, if $C(s)=C(t).$
Clearly, $s$ and $t$ are translation equivalent if and only if $C\cap\left(C+s\right)$
is a translate of $C\cap\left(C+t\right)$. We show that a real number
$t$ is translation equivalent to a rational if and only if some initial
segment of $C(t)$ has a self-similarity property. More precisely,
we show: 
\begin{thm}
\label{Sec-1-Thm:Rational_tau}Let $D$ be sparse and $x\in\left[0,1\right]$
such that $C\left(x\right)$ is not empty. Then $x$ is translation
equivalent to a  rational $t\in\left[0,1\right]$ if and only if there
exists $\varepsilon>0,$ such that $C(x)\cap\left[0,\varepsilon\right]$
is a self-similar set generated by a finite set of similarities $f_{j}(y)=r_{j}y+b_{j},$
where $r_{j}=n^{-q_{j}}$ for some $q_{j}\in\mathbb{N}$. 
\end{thm}
Theorem \ref{Sec-1-Thm:Rational_tau} only requires that the segment
of $C\left(x\right)$ in a neighborhood surrounding zero is self-similar.
However, if $x$ is translation equivalent to a rational, then the
intersection $C(x)\cap[\varepsilon,1]$ cannot be arbitrary, see Theorem
\ref{Sec-1-Thm:Structure-when-t-is-rational}. 

Let $\Delta^{+}:=\Delta\cap\left[0,\infty\right)$ and let $F\left(\Delta^{+}\right)$
be the set of $\alpha$ in the interval $\left[0,1\right]$ such that
\[
\alpha=\sum_{k=1}^{\infty}\alpha_{k}n^{-k},\text{ for some }\alpha_{k}\in\Delta^{+}.
\]
Then $F\left(\Delta^{+}\right)$ is a subset of the set $F^{+}$ of
all $t\in[0,1]$ such that $C\cap\left(C+t\right)$ is non-empty,
and any $t$ in $F^{+}$ is translation equivalent to some $\alpha$
in $F\left(\Delta^{+}\right),$ see Section \ref{Sec-4:Unions-of-Self-Similar}. 

Let $\delta$ be an integer. If $D\cap\left(D+\delta\right)$ is nonempty,
let
\[
D_{\delta}:=D\cap\left(D+\delta\right)-\min\left(D\cap\left(D+\delta\right)\right),
\]
otherwise, let $D_{\delta}$ be the empty set. It follows from Lemma
\ref{Sec-4-Lem:Subset_Equivalence}, that $\alpha$ in $F\left(\Delta^{+}\right)$
is translation equivalent to a rational if and only if there are integers
$k\geq0$ and $q>0,$ such that 
\begin{equation}
D_{\alpha_{j}}\subseteq D_{\alpha_{j}+q}\text{ for all }j>k.\label{Sec-1-eq:equivalent-to-a-rational}
\end{equation}
We say $\alpha$ is \emph{strongly periodic} if there are sets $\widetilde{D}_{\alpha_{j}}$
and $q>0,$ such that 
\[
D_{\alpha_{j}}+\widetilde{D}_{\alpha_{j}}=D_{\alpha_{j+q}}\text{ for all }j>0.
\]
Note this implies (\ref{Sec-1-eq:equivalent-to-a-rational}). We show
in Section \ref{Sub-6.1:Uniform and strongly periodic rationals}
that our notion of strong periodicity is consistent with the one in
\cite{DHW08}, \cite{LYZ11}, \cite{KLD12}, and \cite{ZLL08}, when
$D$ is uniform. 
\begin{thm}
\label{Sec-1-Thm:C(t) self-similar iff }Let $D$ be sparse and $\alpha=0._{n}\alpha_{1}\alpha_{2}\ldots$
be an element in $F\left(\Delta^{+}\right)$. Then $\alpha$ is strongly
periodic if and only if $C\left(\alpha\right)$ is a self-similar
set generated by a finite set of similarities $f_{j}(x)=n^{-q}x+b_{j},$
where $q\in\mathbb{N}$.
\end{thm}
If $D$ is uniform, this was established in \cite{DHW08}, \cite{LYZ11}
when $d_{m}=n-1.$ After we completed this manuscript we received
the preprint \cite{Kon12}, this preprint contains a generalization
of Theorem \ref{Sec-1-Thm:C(t) self-similar iff }, see Remark \ref{Sec-4-Remark:-periodin-not-strongly periodic}.

One part of Theorem \ref{Sec-1-Thm:Rational_tau} is a consequence
of a structure theorem for $C\cap\left(C+x\right),$ when $x$ is
rational. This structure is summarized in the following result. 
\begin{thm}
\label{Sec-1-Thm:Structure-when-t-is-rational}Let $D$ be sparse
and $t\in\left[0,1\right]$ such that $C\cap\left(C+t\right)$ is
not empty. If $t=0._{n}t_{1}\cdots t_{k-p}\overline{t_{k-p+1}\cdots t_{k}}$
for some period $p$ and integer $k\ge p$, then there exists a sparse
digits set $E=\left\{ 0\le e_{1}<e_{2}<\cdots<e_{r}<n^{2p}\right\} $
and corresponding deleted digits Cantor set $C_{n^{2p},E}$ such that
$C\left(t\right)$ consists of a finite number of translates of $\frac{1}{n^{k}}C_{n^{2p},E}$,
the translates of $\frac{1}{n^{k}}C_{n^{2p},E}$ are disjoint, in
fact, the translates of the convex hull of $\frac{1}{n^{k}}C_{n^{2p},E}$
are disjoint. 
\end{thm}
Let $\mathrm{dim_{H}}\left(C\cap\left(C+t\right)\right)$ denote the
Hausdorff measure of $C\cap\left(C+t\right).$ We showed in \cite{PePh11b}
that there are uncountably many $t$ such that the $\mathrm{dim_{H}}\left(C\cap\left(C+t\right)\right)$-dimensional
Hausdorff measure of $C\cap\left(C+t\right)$ is zero or infinity.
For such $t,$ the set $C\cap(C+t)$ is not a finite union of translates
of a self-similar set. In particular, not all real numbers are translation
equivalent to a rational number. We provide a method for constructing
real numbers which are not translation equivalent to any rational,
and thus are not a finite, disjoint union of self-similar sets. In
particular, we show that if $D$ is uniform and $t\in C_{n,D}$ is
irrational, then $t$ is not translation equivalent to any rational.

The structure of uniform deleted digits Cantor sets allows us to prove
the following special case of Theorem \ref{Sec-1-Thm:Rational_tau}:
\begin{thm}
\label{Sec-1-Thm:Uniform union of self-similar sets}Let $D$ be uniform
and $x\in\left[0,1\right]$ such that $C\cap\left(C+x\right)$ is
not empty. There exists a rational $t\in\left[0,1\right]$ such that
$C(x)=C(t)$ if and only if $C\left(x\right)$ is the finite, disjoint
union of self-similar sets.
\end{thm}
We show in Section \ref{Sec-6.2:beta expansions} that our results
extend to a class of $\beta-$expansions with non-uniform digits sets.
The papers \cite{ZLL08} and \cite{KLD12} consider $\beta-$expansions
with uniform digit sets, but they allow a different class of $\beta$'s
than we do. 

Other properties of intersections of Cantor sets have recently been
studied, see e.g., \cite{DaHu95}, \cite{KePe91}, \cite{Kr99}, \cite{MSS09},\cite{Phi11},
\cite{PePh11a}, and \cite{PePh11b}.

We refer the reader to \cite{Fal85} for background information on
Hausdorff dimension, Hausdorff measure and self-similar sets.

\subsection{Outline}

In Section \ref{Sec-2:Old-Stuff} we summarize the construction of
$C\cap\left(C+t\right)$ in our analysis. More details can be found
in \cite{PePh11b} where this construction was used to investigate
the Hausdorff measure of $C\cap\left(C+t\right)$. A related construction
was used in \cite{PePh11a} to investigate the Hausdorff dimension
of $C\cap\left(C+t\right)$. 

In Section \ref{Sec-3:Rational t} we investigate some aspects of
translation equivalence leading to a proof of Theorem \ref{Sec-1-Thm:Structure-when-t-is-rational}.
We calculate the Hausdorff measure of $C\cap\left(C+t\right)$ for
some $C$ and $t$ and apply our methods to situations when $D$ is
not sparse.

In Section \ref{Sec-4:Unions-of-Self-Similar} we resume our analysis
of translation equivalence leading to a proof of Theorem \ref{Sec-1-Thm:Rational_tau}
and to a proof of Theorem \ref{Sec-1-Thm:C(t) self-similar iff }

In Section \ref{Sec-5:Non-Self-Similar Sets} we associate an uncountable
family of irrationals that are not translation equivalent to a rational
to any $t$ such that $C(t)$ is not finite. 

Finally, in Section \ref{Sec-6:Uniform Sets}, we focus on uniform
sets and discuss the relationship between strong periodicity and translation
equivalence. We extend the definition of strongly periodic rationals
to an arbitrary digits set $D$ and show, if $D$ is uniform, then
Theorem \ref{Sec-1-Thm:Rational_tau} holds with $r_{j}=n^{-q_{j}}$
replaced by $r_{j}>0.$ We prove that our results hold for certain
$\beta$-expansions with non-uniform digits sets.

\section{A Construction of $C\cap(C+t)$\label{Sec-2:Old-Stuff}}

In this section we assume $n\geq3$ is given and that $D=\left\{ d_{k}\mid k=1,2,\ldots,m\right\} ,$
$2\leq m<n$ is a digits set. We demonstrate a natural method for
constructing $C=C_{n,D}$, which forms the basis for our analysis
of $C\cap\left(C+t\right).$ The results in this section are proven
in \cite{PePh11b}, but we summarize the relevant parts of \cite{PePh11b}
here for the convenience of the reader.

In order to avoid trivial cases where $C\cap\left(C+t\right)$ is
empty, define 
\[
F:=\left\{ t\mid C\cap\left(C+t\right)\neq\varnothing\right\} .
\]
It is easy to see that $F=C-C=\{x-y\mid x,y\in C\}$ and consequently,
$F$ is compact. Since $C\cap\left(C-t\right)$ is a translate of
$C\cap\left(C+t\right)$ it is sufficient to consider $t\geq0$ and
$F=\left(-F^{+}\right)\cup F^{+}$ where $F^{+}:=F\cap[0,\infty)$. 

The middle thirds Cantor set is often constructed by beginning with
the closed interval $C_{0}=[0,1]$ and, inductively, for $k\geq0$,
obtaining $C_{k+1}$ from $C_{k}$ by removing the open middle of
each interval in $C_{k}.$ In general, $C=C_{n,D}$ can be constructed
in a similar manner. The \emph{refinement} of an interval $[a,b]$
is the set
\[
\bigcup_{j=1}^{m}\left[a+\frac{d_{j}}{n}\left(b-a\right),a+\frac{d_{j}+1}{n}\left(b-a\right)\right].
\]
Let $C_{0}$ be the closed unit interval $[0,1]$ and inductively,
for $k\geq0,$ obtain $C_{k+1}$ from $C_{k}$ by refining each $n$-ary
interval in $C_{k}.$ Then, $C_{k}:=\left\{ 0._{n}x_{1}x_{2}\ldots\mid x_{j}\in D\text{ for }1\le j\le k\right\} $,
$C_{k+1}\subset C_{k}$ for all $k$, and 
\begin{equation}
C=C_{n,D}=\bigcap_{k=0}^{\infty}C_{k}=\left\{ 0._{n}x_{1}x_{2}\ldots\mid x_{j}\in D\text{ for all }1\le j\right\} .\label{Sec-2-Eq:C-construction-by-intersections}
\end{equation}

For any integer $h$, we say that an interval $J^{\left(h\right)}=\frac{1}{n^{k}}\left(C_{0}+h\right)$
is an \emph{$n$-ary interval} of length $\frac{1}{n^{k}}$\emph{.}
We will simply say $n$-ary interval when $k$ is understood from
the context. In particular, if $U$ is a compact set, the phrase \emph{an
$n$-ary interval of $U$} refers to an $n$-ary interval of length
$\frac{1}{n^{k}}$ contained in $U$ where $k$ is the smallest such
$k$. Note, $C_{k}$ consists of $m^{k}$ disjoint $n$-ary intervals. 

For a fixed $t=0._{n}t_{1}t_{2}\ldots$ in $[0,1]$, our analysis
of $C\cap\left(C+t\right)$ has three ingredients: (\emph{i}) It follows
from (\ref{Sec-2-Eq:C-construction-by-intersections}) that 
\begin{equation}
C\cap\left(C+t\right)=\bigcap_{k=0}^{\infty}\left(C_{k}\cap\left(C_{k}+t\right)\right).\label{Sec-2-Eq:C_t-as-intersection-of- C_ks}
\end{equation}
(\emph{ii}) There is a relationship, see Lemma \ref{Sec-2-Lem:Interval_Count},
between $C_{k}\cap\left(C_{k}+t\right)$ and $C_{k}\cap\left(C_{k}+\left\lfloor t\right\rfloor _{k}\right)$,
where $\left\lfloor t\right\rfloor _{k}=0._{n}t_{1}t_{2}\ldots t_{k}$.
(\emph{iii}), the structure of the set $C_{k}\cap\left(C_{k}+\left\lfloor t\right\rfloor _{k}\right)$
is related to the structure of the set $C_{k+1}\cap\left(C_{k+1}+\left\lfloor t\right\rfloor _{k+1}\right)$,
see the definition of $\sigma_{t}$ and Lemma \ref{Sec-2-Lem:Sigma-property},
below.

Since $\left\lfloor t\right\rfloor _{k}=\frac{h}{n^{k}}$ for some
integer $h$, then $C_{k}+\left\lfloor t\right\rfloor _{k}$ also
consists of $m^{k}$ disjoint $n$-ary intervals. Thus, an $n$-ary
interval $J^{\left(h\right)}\subset C_{k}$ may interact with $C_{k}+\left\lfloor t\right\rfloor _{k}$
in combinations of only four cases: we say $J^{\left(h\right)}$ is
in the \emph{interval case} if $J^{\left(h\right)}$ is also an interval
of $C_{k}+\left\lfloor t\right\rfloor _{k}$, the \emph{potential
interval case} if $J^{\left(h\right)}+\frac{1}{n^{k}}$ is an interval
of $C_{k}+\left\lfloor t\right\rfloor _{k}$, the \emph{potentially
empty case} if $J^{\left(h\right)}-\frac{1}{n^{k}}$ is an interval
of $C_{k}+\left\lfloor t\right\rfloor _{k}$, and the \emph{empty
case} if $J^{\left(h\right)}\cap\left(C_{k}+\left\lfloor t\right\rfloor _{k}\right)$
is empty. 
\begin{rem}
\label{Sec-2-Rem:Finite n-ary Representations}According to Theorem
3.1 of \cite{PePh11b}, if $D$ is any digits set and $t\in\left[0,1\right]$
admits a finite $n$-ary representation, then $C\cap\left(C+t\right)=A\cup B$
where $A$ is either empty or a finite, disjoint collection of sets
of the form $\frac{1}{n^{k}}\left(C+h\right)$ for some integers $k$
and $h$, and $B$ is either empty or a finite collection of points.
For these reasons, we focus on real numbers which do not admit finite
$n$-ary representation.
\end{rem}
It is important to note that only interval and potential interval
cases can contribute points to $C\cap\left(C+t\right)$ whenever $t$
does not admit a finite $n$-ary representation.
\begin{lem}
\label{Sec-2-Lem:empty-can-be-ignored}Suppose $0<t-\left\lfloor t\right\rfloor _{k}<\frac{1}{n^{k}}$
for some $k$. If $J$ is an $n$-ary interval in $C_{k}$ and $J$
is either in the potentially empty or the empty case, then $J\cap\left(C_{k}+t\right)$
is empty. In particular, the intersection $J\cap C\cap\left(C+t\right)$
is empty. 
\end{lem}
It is possible for $J^{\left(h\right)}$ to be both in the interval
case and potentially empty case, or to be both in the potential interval
case and in the potentially empty case. However, the intersections
corresponding to the potentially empty cases do not contribute points
to $C\cap\left(C+t\right)$ when $0<t-\left\lfloor t\right\rfloor _{k}$
and we will not identify these cases with special terminology.

We introduce a function whose values tell us whether $C_{k}\cap\left(C_{k}+\left\lfloor t\right\rfloor _{k}\right)$
contains interval cases, potential interval cases, both, or neither.
Since $C_{0}\cap\left(C_{0}+\left\lfloor t\right\rfloor _{0}\right)=\left[0,1\right]$
consists of a single interval case, then we can examine $C_{k}\cap\left(C_{k}+\left\lfloor t\right\rfloor _{k}\right)$
using induction. Let $i:=\sqrt{-1}$ and let $\xi:\left\{ 0,\pm1,i\right\} \times\left\{ 0,\pm1,\ldots,\pm\left(n-1\right)\right\} \rightarrow\left\{ 0,\pm1,\pm i\right\} $
be determined by
\begin{align*}
\xi\left(0,h\right) & :=0\\
\xi\left(1,h\right) & :=\begin{cases}
1 & \text{if }\left|h\right|\text{ is in }\Delta\text{ but not in }\Delta-1\\
-1 & \text{if }\left|h\right|\text{ is in }\Delta-1\text{ but not in }\Delta\\
i & \text{if }\left|h\right|\text{ is both in }\Delta\text{ and }\Delta-1\\
0 & \text{otherwise}
\end{cases}\\
\xi\left(-1,h\right) & :=\begin{cases}
-1 & \text{if }\left|h\right|\text{ is in }n-\Delta\text{ but not in }n-\Delta-1\\
1 & \text{it }\left|h\right|\text{ is in }n-\Delta-1\text{ but not in }n-\Delta\\
-i & \text{if }\left|h\right|\text{ is both in }n-\Delta\text{ and in }n-\Delta-1\\
0 & \text{otherwise}
\end{cases}\\
\xi\left(i,h\right) & :=\begin{cases}
-i & \text{if }\left|h\right|\text{ is in }\Delta\cup\left(n-\Delta\right)\text{ but not in }\left(\Delta-1\right)\cup\left(n-\Delta-1\right)\\
i & \text{if }\left|h\right|\text{ is in }\left(\Delta-1\right)\cup\left(n-\Delta-1\right)\text{ but not in }\Delta\cup\left(n-\Delta\right)\\
1 & \text{if }\left|h\right|\text{ is both in }\Delta\cup\left(n-\Delta\right)\text{ and in }\left(\Delta-1\right)\cup\left(n-\Delta-1\right)\\
0 & \text{otherwise}.
\end{cases}
\end{align*}
The function $\xi\left(z,h\right)$ is completely determined by $D$
and $n$. Let $\sigma_{t}:\mathbb{N}_{0}\to\left\{ 0,\pm1,i\right\} $
be determined by
\begin{align*}
\sigma_{t}(0) & :=1\text{ and inductively }\\
\sigma_{t}\left(k+1\right) & :=\xi\left(\sigma_{t}\left(k\right),t_{k+1}\right)\cdot\sigma_{t}\left(k\right)\text{ for }k\geq0.
\end{align*}
By construction of $\xi$ we have $\sigma_{t}\left(k\right)\in\left\{ 0,\pm1,i\right\} $
for all $k\geq0.$ Compared to \cite{PePh11a} the present definition
of $\xi$ uses $|h|$ in place of $h.$ This is to anticipate a variant
needed in Section \ref{Sec-4:Unions-of-Self-Similar}. 
\begin{lem}
\label{Sec-2-Lem:Sigma-property}Let $t=0._{n}t_{1}t_{2}\cdots$ be
some point in $[0,1].$ Then $C_{k}\cap\left(C_{k}+\left\lfloor t\right\rfloor _{k}\right)$
contains interval cases but no potential interval cases iff $\sigma_{t}\left(k\right)=1$,
potential interval cases but no interval cases iff $\sigma_{t}\left(k\right)=-1$,
both interval and potential interval cases iff $\sigma_{t}\left(k\right)=i$,
and neither interval cases nor potential interval cases iff $\sigma_{t}\left(k\right)=0$.
\end{lem}
Lemma \ref{Sec-2-Lem:Sigma-property} allows us to describe $F$ in
terms of $\sigma_{t},$ when $D$ is sparse. 
\begin{lem}
\label{Sec-2-Lem:Sparcity_Requirement}Let $C=C_{n,D}$ be a deleted
digits Cantor set. Then $D$ is sparse iff 
\[
F^{+}=\left\{ t\in[0,1]\mid\sigma_{t}\left(k\right)=\pm1\text{ for all }k\in\mathbb{N}\right\} .
\]

\end{lem}
Let $\#E$ denote the number of elements in a finite set $E$. Define
$\mu_{t}(0):=1$ and inductively
\[
\mu_{t}\left(k+1\right)=\begin{cases}
\mu_{t}(k)\cdot\#\left(D-t_{k+1}\right)\cap\left(D\cup\left(D+1\right)\right) & \text{ if }\sigma_{t}\left(k\right)=1\\
\mu_{t}(k)\cdot\#\left(D-n+t_{k+1}\right)\cap\left(D\cup\left(D-1\right)\right) & \text{ if }\sigma_{t}\left(k\right)=-1
\end{cases}
\]
The function $\mu_{t}$ also depends on $n$ and $D,$ but we suppress
this dependence in the notation. The function $\mu_{t}$ provides
a method for counting the number of intervals contained in $C_{k}\cap\left(C_{k}+t\right)$.
\begin{lem}
\label{Sec-2-Lem:Interval_Count}Let $C=C_{n,D}$ be given. Suppose
$t\in F^{+}$ does not admit a finite $n$-ary representation and
$\sigma_{t}(k)=\pm1$ for all $k\geq0$. Then $C_{k}\cap\left(C_{k}+t\right)$
is a union of $\mu_{t}(k)$ intervals, each of length 
\begin{align*}
\ell_{k}:= & \begin{cases}
\frac{1}{n^{k}}-\left(t-\left\lfloor t\right\rfloor _{k}\right) & \text{ when }\sigma_{t}\left(k\right)=1\\
t-\left\lfloor t\right\rfloor _{k} & \text{ when }\sigma_{t}\left(k\right)=-1
\end{cases}.
\end{align*}

\end{lem}
While $\mu_{t}\left(k\right)$ provides an upper bound to the number
of intervals of length $\ell_{k}$ required to cover $C_{k}\cap\left(C_{k}+t\right)$,
it is important to know that each of these intervals contains points
in $C\cap\left(C+t\right)$.
\begin{lem}
\label{Sec-2-Lem:Nothing_Goes_Away}Let $C=C_{n,D}$ be given. Suppose
$t\in F^{+}$ does not admit a finite $n$-ary representation and
$\sigma_{t}(k)=\pm1$ for all $k\geq0$. For each $k$, every $n$-ary
interval of $C_{k}$ in the interval or potential interval case contains
points of $C\cap\left(C+t\right)$.
\end{lem}
If $D$ is not sparse, then some interval or potential interval cases
may not lead to points in $C\cap\left(C+t\right)$, see Example \ref{Sec-3-Ex:Self-Similar with sigma=00003Di}.

\section{\label{Sec-3:Rational t}Real Values $t$.}

In this section, we prove Theorem \ref{Sec-1-Thm:Structure-when-t-is-rational}.
Part of Theorem \ref{Sec-1-Thm:Rational_tau} is an immediate consequence.
The other part of Theorem \ref{Sec-1-Thm:Rational_tau} is proved
in Section \ref{Sec-4:Unions-of-Self-Similar}. 

As in Section \ref{Sec-2:Old-Stuff}, many of the results of this
section only require that $t\in F$ and $\sigma_{t}\left(k\right)=\pm1$
for all $k$. This allows us to apply our results when $D$ is not
sparse for specific values of $t$, see Section \ref{Sub-3.4:Non-sparse D}.
On the other hand, if $D$ is sparse the condition $\sigma_{t}\left(k\right)=\pm1$
follows immediately from Lemma \ref{Sec-2-Lem:Sparcity_Requirement}
for all $t$ in $F$.

\subsection{\label{Sub-3.1:n-ary Translation Equivalence}Translation Equivalence
of $n$-ary representations.}

We begin by investigating the structure of $C\cap\left(C+t\right)$
for an arbitrary value $t$ in $F$. Lemma \ref{Sec-3-Lem:Contained_(t-t_k)}
describes how the structure of $C\cap\left(C+t\right)$ is related
to the $n$-ary representation $t=0._{n}t_{1}t_{2}\ldots$. 
\begin{lem}
\label{Sec-3-Lem:Contained_(t-t_k)}Let $C=C_{n,D}$ be given and
$t\in F^{+}$ such that $t$ does not admit finite $n$-ary representation
and $\sigma_{t}\left(k\right)=\pm1$ for all $k\in\mathbb{N}_{0}$.
Then $C\cap\left(C+t\right)$ is a union of $\mu_{t}(k)$ disjoint
copies of
\[
\frac{1}{n^{k}}\left[C\cap\left(C+n^{k}\left(t-\left\lfloor t\right\rfloor _{k}\right)\right)\right]\text{ when }\sigma_{t}\left(k\right)=1
\]
and of $\mu_{t}(k)$ disjoint copies of 
\[
\frac{1}{n^{k}}\left[C\cap\left(C-1+n^{k}\left(t-\left\lfloor t\right\rfloor _{k}\right)\right)\right]\text{ when }\sigma_{t}\left(k\right)=-1.
\]
\end{lem}
\begin{proof}
Any $n$-ary interval in $C_{k}$ is of the form $J^{(h)}=\frac{1}{n^{k}}\left(C_{0}+h\right)$
for some $h\in\mathbb{Z}.$ Let $J_{j}^{(h)}:=\frac{1}{n^{k}}\left(C_{j}+h\right).$
Then
\begin{align*}
\frac{1}{n^{k}}\left(C\cap\left(C+n^{k}\left(t-\left\lfloor t\right\rfloor _{k}\right)\right)\right)+\frac{h}{n^{k}} & =\bigcap_{j=1}^{\infty}\left(J_{j}^{(h)}\cap\left(J_{j}^{(h)}+\left(t-\left\lfloor t\right\rfloor _{k}\right)\right)\right)\text{ and}\\
\frac{1}{n^{k}}\left(C\cap\left(C-1+n^{k}\left(t-\left\lfloor t\right\rfloor _{k}\right)\right)\right)+\frac{h}{n^{k}} & =\bigcap_{j=1}^{\infty}\left(J_{j}^{(h)}\cap\left(J_{j}^{(h)}-\frac{1}{n^{k}}+\left(t-\left\lfloor t\right\rfloor _{k}\right)\right)\right)
\end{align*}
 Now $J^{(h)}\subseteq C_{k}$ implies $J_{j}^{(h)}\subset C_{k+j}$
for all $j,$ since $J_{j}^{(h)}$ is obtained from $J^{(h)}$ by
repeated refinement. 

According to Lemma \ref{Sec-2-Lem:empty-can-be-ignored}, only $n$-ary
intervals in $C_{k}$ that are in the interval case or the potential
intervals case can have points in common with $C\cap\left(C+t\right).$
By Lemma \ref{Sec-2-Lem:Interval_Count} there are $\mu_{t}(k)$ $n$-ary
intervals $J^{(h)}\subseteq C_{k}$ in the interval or the potential
interval case. 

Suppose $\sigma_{t}\left(k\right)=1$. Then $J^{(h)}$ is in the interval
case by Lemma \ref{Sec-2-Lem:Sigma-property}. Hence $J^{(h)}\subset C_{k}+\left\lfloor t\right\rfloor _{k}$
and therefore $J^{(h)}+\left(t-\left\lfloor t\right\rfloor _{k}\right)\subset C_{k}+t.$
By repeated refinement $J_{j}^{(h)}+\left(t-\left\lfloor t\right\rfloor _{k}\right)\subset C_{k+j}+t.$
Consequently, 
\begin{align*}
\bigcap_{j=1}^{\infty}\left(J_{j}^{(h)}\cap\left(J_{j}^{(h)}+\left(t-\left\lfloor t\right\rfloor _{k}\right)\right)\right) & \subseteq\bigcap_{j=1}^{\infty}\left(C_{k+j}\cap\left(C_{k+j}+t\right)\right)\\
 & =C\cap(C+t).
\end{align*}

Suppose $\sigma_{t}\left(k\right)=-1$. Then $J^{(h)}$ is in the
potential interval case by Lemma \ref{Sec-2-Lem:Sigma-property}.
Hence $J^{(h)}-\frac{1}{n^{k}}\subset C_{k}+\left\lfloor t\right\rfloor _{k}$
and therefore $J^{(h)}-\frac{1}{n^{k}}+\left(t-\left\lfloor t\right\rfloor _{k}\right)\subset C_{k}+t.$
By repeated refinement $J_{j}^{(h)}-\frac{1}{n^{k}}+\left(t-\left\lfloor t\right\rfloor _{k}\right)\subset C_{k+j}+t.$
Consequently, 
\begin{align*}
\bigcap_{j=1}^{\infty}\left(J_{j}^{(h)}\cap\left(J_{j}^{(h)}-\frac{1}{n^{k}}+\left(t-\left\lfloor t\right\rfloor _{k}\right)\right)\right) & \subseteq\bigcap_{j=1}^{\infty}\left(C_{k+j}\cap\left(C_{k+j}+t\right)\right)\\
 & =C\cap(C+t).
\end{align*}

Conversely, suppose $x\in C\cap\left(C+t\right)=\bigcap_{k=1}^{\infty}C_{k}\cap\left(C_{k}+t\right)$.
Let $k\in\mathbb{N}_{0}$ be arbitrary and $J^{(h)}\subset C_{k}$
denote the $n$-ary interval such that $x\in J^{(h)}$ for some $h$
by Lemma \ref{Sec-2-Lem:empty-can-be-ignored} and $J^{(h)}$ is in
interval or potential interval case. Let 
\[
I_{k}:=\left\{ i\mid J^{\left(i\right)}\subset C_{k}\text{ is \ensuremath{n}-ary and in the interval or potential interval case}\right\} .
\]
 If $J^{\left(j\right)}=J^{\left(i\right)}-\frac{1}{n^{k}}$ for some
$j,i\in I_{k}$ then either $J^{\left(i\right)}$ or $J^{\left(j\right)}$
is in both the interval and potential interval cases, which contradicts
that $\sigma_{t}\left(k\right)=\pm1$. Thus, any $J^{\left(i\right)}\subset C_{k}\cap\left(C_{k}+\left\lfloor t\right\rfloor _{k}\right)$
such that $i\neq j$ is at least a distance of $\frac{1}{n^{k}}$
from $J^{\left(j\right)}$ and $J^{\left(j\right)}\cap J^{\left(i\right)}=\varnothing$.

Suppose $\sigma_{t}\left(k\right)=1$ so that $J^{(i)}$ is in the
interval case for all $i\in I_{k}$. Then $C_{k}\cap\left(C_{k}+t\right)=\bigcup_{i\in I}J^{\left(i\right)}\cap\left(J^{\left(i\right)}+\left(t-\left\lfloor t\right\rfloor _{k}\right)\right)$
so that $x\in J^{\left(h\right)}\cap C_{k}\cap\left(C_{k}+t\right)=J^{\left(h\right)}\cap\left(J^{\left(h\right)}+\left(t-\left\lfloor t\right\rfloor _{k}\right)\right)$.
Furthermore, $J^{\left(h\right)}\cap C_{k+j}=J_{j}^{\left(h\right)}$
for each $j>0$ by construction of $C_{k}$ so that 
\[
x\in\bigcap_{j=1}^{\infty}\left(J^{\left(h\right)}\cap C_{k+j}\cap\left(C_{k+j}+t\right)\right)=\bigcap_{j=1}^{\infty}\left(J_{j}^{\left(h\right)}\cap\left(J_{j}^{\left(h\right)}+\left(t-\left\lfloor t\right\rfloor _{k}\right)\right)\right).
\]
Since $x$ is arbitrary, then $C\cap\left(C+t\right)$ is a subset
of the disjoint union
\[
\bigcup_{h\in I_{k}}\left(\bigcap_{j=1}^{\infty}\left(J_{j}^{\left(h\right)}\cap\left(J_{j}^{\left(h\right)}+\left(t-\left\lfloor t\right\rfloor _{k}\right)\right)\right)\right).
\]

The case $\sigma_{t}\left(k\right)=-1$ is obtained by replacing $J_{j}^{(h)}+\left(t-\left\lfloor t\right\rfloor _{k}\right)$
by $J_{j}^{(h)}-\frac{1}{n^{k}}+\left(t-\left\lfloor t\right\rfloor _{k}\right)$
above. This completes the proof.
\end{proof}
According to Lemma \ref{Sec-2-Lem:Interval_Count}, if $\sigma_{t}\left(j\right)=\pm1$
for all $j$ then $C_{k}\cap\left(C_{k}+\left\lfloor t\right\rfloor _{k}\right)$
is the disjoint union of $\mu_{t}\left(k\right)$ $n$-ary intervals
of length $\frac{1}{n^{k}}$. Using the definition $I_{k}$ from the
proof of Lemma \ref{Sec-3-Lem:Contained_(t-t_k)}, $C_{k}\cap\left(C_{k}+\left\lfloor t\right\rfloor _{k}\right)=\bigcup_{h\in I_{k}}J^{\left(h\right)}$
and, for each $h\in I_{k}$, the corresponding interval $J^{\left(h\right)}$
refines to a ``small'' Cantor set intersected with its real translate.
Lemma \ref{Sec-3-Lem:Contained_(t-t_k)} shows that, for each $k$,
the translation value directly depends on digits $t_{j}$ for $j>k$,
and the spacing of intervals $J^{\left(h\right)}$ depends on $\left\lfloor t\right\rfloor _{k}=0._{n}t_{1}t_{2}\ldots t_{k}$.
The requirement that $\sigma_{t}\left(k\right)=\pm1$ guarantees that
the intervals $J^{\left(h\right)}$ are disjoint. These results follow
from the analysis in Section \ref{Sec-2:Old-Stuff}.

The next few lemmas establish some properties of translation equivalence.
The first of these result allows us to calculate limits of sequences
of the form $C\cap\left(C+x_{j}\right).$ 

Let $\mathscr{H}^{s}(K)$ denote the $s$-dimensional Hausdorff measure
of a set $K.$ If $D$ is sparse and $0<\beta,y<\infty$ are arbitrary
real numbers, then the set 
\[
F_{\beta,y}:=\left\{ t\mid m^{-2\beta}y\le\mathscr{H}^{\beta\log_{n}\left(m\right)}\left(C\cap\left(C+t\right)\right)\le y\right\} 
\]
 is dense in $F$, see \cite{PePh11b}. Thus, the mapping $t\mapsto\mathscr{H}^{s}\left(C\cap\left(C+t\right)\right)$
is everywhere discontinuous on $F$. In general, if $\left\{ x_{j}\right\} _{j=0}^{\infty}$
is a sequence in $F^{+}$ which converges to $x$, then $\lim_{j\to\infty}\left(C\cap\left(C+x_{j}\right)\right)$
need not equal $C\cap\left(C+x\right)$, even when the limit exists
with respect to the Hausdorff metric. However, with suitable restrictions
on the sequence $\left\{ x_{j}\right\} $ we show that the sets $C\cap\left(C+x_{j}\right)$
do converge.
\begin{lem}
\label{Sec-3-Lem:Sequence of Sets}Let $C=C_{n,D}$ and $x_{0}$ be
given. Suppose $\left\{ x_{j}\right\} _{j=0}^{\infty}$ is a sequence
in $F^{+}$ converging to some real number $x$. If

\begin{minipage}[t]{1\columnwidth}%
\begin{enumerate}
\item $\sigma_{x_{j}}\left(k\right)=\pm1$ for all $j,k\in\mathbb{N}_{0}$,
\item for each $j$ there exists $c_{j}\in\mathbb{R}$ such that $C\cap\left(C+x_{j}\right)=\left(C\cap\left(C+x_{0}\right)\right)+c_{j}$
and 
\item the sequence $c_{j}$ converges to $c\in\mathbb{R}$,\medskip{}
\end{enumerate}
\end{minipage}

then $C\cap\left(C+x\right)=\left(C\cap\left(C+x_{0}\right)\right)+c$.\end{lem}
\begin{proof}
Note that $x\in F^{+}$ by compactness. Furthermore, the sequence
of compact sets $\left(C\cap\left(C+x_{0}\right)\right)+c_{j}$ converges
in the Hausdorff metric so that
\begin{align*}
\lim_{j\to\infty}\left(C\cap\left(C+x_{j}\right)\right) & =\left(C\cap\left(C+x_{0}\right)\right)+c.
\end{align*}

We must show that $C\cap\left(C+x\right)=\left(C\cap\left(C+x_{0}\right)\right)+c$.
The result is trivial if $x=x_{j}$ for some $j$, so suppose $x\neq x_{j}$
for all $j$.

Let $y\in\left(C\cap\left(C+x_{0}\right)\right)+c$ be arbitrary.
For each $j$, choose $y_{j}\in C\cap\left(C+x_{j}\right)$ such that
$y_{j}$ converges to $y$. Thus, $y_{j}$ is a sequence of $C$ so
that $y\in C$ and $\lim_{j\to\infty}\left\{ C+x_{j}\right\} =C+x$
converges in the hausdorff metric so that $y\in\left(C+x\right)$.
Hence $\left(C\cap\left(C+x_{0}\right)\right)+c\subseteq C\cap\left(C+x\right)$.

Let $y\in C\cap\left(C+x\right)$ be arbitrary. For each $j\in\mathbb{N}$,
choose $N_{j}\in\mathbb{N}$ such that $\left(\frac{1}{n}\right)^{N_{j}+1}\le\left|x-x_{j}\right|<\left(\frac{1}{n}\right)^{N_{j}}$.
Thus, $\left\lfloor x\right\rfloor _{N_{j}}=\left\lfloor x_{j}\right\rfloor _{N_{j}}$
and $C_{N_{j}}\cap\left(C_{N_{j}}+\left\lfloor x\right\rfloor _{N_{j}}\right)=C_{N_{j}}\cap\left(C_{N_{j}}+\left\lfloor x_{j}\right\rfloor _{N_{j}}\right)$.
Let $J$ be the $n$-ary interval of $C_{N_{j}}$ which contains $y$.
According to Lemma \ref{Sec-2-Lem:Nothing_Goes_Away}, $J$ contains
points of $C\cap\left(C+x_{j}\right)$ so choose $y_{j}\in J\cap C\cap\left(C+x_{j}\right)$.
Since $J$ has length $\left(\frac{1}{n}\right)^{N_{j}}$ then $\left|y-y_{j}\right|\le\left(\frac{1}{n}\right)^{N_{j}}$. 

Thus, we can construct a sequence $\left\{ y_{j}\right\} $ such that
$y_{j}\in C\cap\left(C+x_{j}\right)$ for each $j$. Since $x_{j}\rightarrow x$,
then $N_{j}\to\infty$ and $y_{j}$ converges to $y$. Hence, $y\in\lim_{j\to\infty}\left\{ C\cap\left(C+x_{j}\right)\right\} $
and $C\cap\left(C+x\right)=\left(C\cap\left(C+x_{0}\right)\right)+c$.\end{proof}
\begin{cor}
Let $\left\{ x_{j}\right\} _{j=0}^{\infty}$ be a sequence in $F^{+}$
such that $x_{j}$ converges to $x$, $\sigma_{x_{j}}\left(k\right)=\pm1$
for all $k$, $C\cap\left(C+x_{j}\right)=\left(C\cap\left(C+x_{0}\right)\right)+c_{j}$
for each $j$, and the sequence $c_{j}$ converges to $c$. Then $c_{j}\in F$
for all $j\in\mathbb{N}_{0}$.\end{cor}
\begin{proof}
Let $j$ be arbitrary. Then $C\cap\left(C+x_{j}\right)=\left(C+c_{j}\right)\cap\left(C+x_{0}+c_{j}\right)$
so that any element $y\in C\cap\left(C+x_{j}\right)$ is contained
in both $C$ and $\left(C+c_{j}\right)$. Thus $c_{j}\in F$ for all
$j$ and $c\in F$ by compactness. 
\end{proof}
We now show that when $t$ is rational with period $p$, then $\sigma_{t}$
is also periodic with period $p$ or $2p$.
\begin{lem}
\label{Sec-3-Lem:Sigma_period_q}Let $C_{n,D}$ be given. Suppose
$t\in F^{+}$ does not admit finite $n$-ary representation, $\sigma_{t}\left(k\right)=\pm1$
for all $k\in\mathbb{N}$, and $t=0._{n}t_{1}\cdots t_{k}\overline{t_{k+1}\cdots t_{k+p}}$
for some integer $k\ge0$ and period $p$. Then $\sigma_{t}$ has
period $p$ or $2p$.\end{lem}
\begin{proof}
Suppose $\sigma_{t}\left(k+1\right)=\sigma_{t}\left(k+p+1\right)$.
By induction, for any $j>k$, 
\begin{align*}
\sigma_{t}\left(j+p+1\right) & =\xi\left(\sigma_{t}\left(j+p\right),t_{j+p+1}\right)\cdot\sigma_{t}\left(j+p\right)\\
 & =\xi\left(\sigma_{t}\left(j\right),t_{j+1}\right)\cdot\sigma_{t}\left(j\right)\\
 & =\sigma_{t}\left(j+1\right).
\end{align*}

Therefore, $\sigma_{t}\left(j\right)=\sigma_{t}\left(j+p\right)$
for all $j>k$ and $\sigma_{t}\left(k\right)$ has period $p$. 

Suppose $\sigma_{t}\left(k+1\right)=-\sigma_{t}\left(k+p+1\right)$.
If $\sigma_{t}\left(k+p+1\right)=\sigma_{t}\left(k+2p+1\right)$ then
$\sigma_{t}\left(j\right)=\sigma_{t}\left(j+p\right)$ for $j>k+p$
and $\sigma_{t}$ has period $p$ by the argument above. Otherwise,
$\sigma_{t}\left(k+1\right)=\sigma_{t}\left(k+2p+1\right)$. By induction,
for any $j>k$, 
\begin{align*}
\sigma_{t}\left(j+2p+1\right) & =\xi\left(\sigma_{t}\left(j+2p\right),t_{j+2p+1}\right)\cdot\sigma_{t}\left(j+2p\right)\\
 & =\xi\left(\sigma_{t}\left(j\right),t_{j+1}\right)\cdot\sigma_{t}\left(j\right)\\
 & =\sigma_{t}\left(j+1\right).
\end{align*}
Hence, $\sigma_{t}\left(j\right)$ has period $2p$.
\end{proof}
It follows from the next lemma that, if $D$ is sparse, then any $t$
in $F^{+}$ is translation equivalent to an $s$ in $F^{+}$ such
that $\sigma_{s}(k)=1$ for all $k.$ That is, all intervals in $C_{k}\cap\left(C_{k}+\left\lfloor s\right\rfloor _{k}\right)$
are in the interval case. We need Lemma \ref{Sec-3-Lem:Sigma_period_q}
to show that, if $t$ is rational, then the $s$ we construct is also
rational. 
\begin{lem}
\label{Sec-3-Lem:Constructing y Equiv t With Sigma_y=00003D1}Let
$C=C_{n,D}$ be given. Suppose $t\in F^{+}$ does not admit finite
$n$-ary representation and $\sigma_{t}\left(k\right)=\pm1$ for all
$k\in\mathbb{N}$. Then there exists $y\in F^{+}$ such that $\sigma_{y}\left(k\right)=1$
for all $k\in\mathbb{N}$ and $C\cap\left(C+t\right)=\left(C\cap\left(C+y\right)\right)+c$
for some $c\in\mathbb{R}$. If $t$ is rational, then $y$ is also
rational.\end{lem}
\begin{proof}
Let $t\in F$ be arbitrary. For a real sequence $\left\{ x_{i}\right\} _{i=0}^{\infty}$,
let $x_{i}:=0._{n}x_{i1}x_{i2}\ldots$ denote the $n$-ary representation.
Let $x_{0}:=t$ so that $x_{0j}=t_{j}$ and $\sigma_{x_{0}}\left(j\right)=\pm1$
for all $j\in\mathbb{N}$.

We will construct sequences $\left\{ x_{i}\right\} _{i=0}^{\infty}$
and $\left\{ \sum_{j=0}^{i}c_{k_{j}}\right\} _{i=0}^{\infty}$ which
satisfy Lemma \ref{Sec-3-Lem:Sequence of Sets} and then show that
$y:=\lim_{i\to\infty}x_{i}$ satisfies our conditions. Let $c_{k_{0}}:=0$
so that $C\cap\left(C+x_{0}\right)=\left(C\cap\left(C+t\right)\right)+c_{k_{0}}$
and the translation condition is true for $i=0$. Suppose $C\cap\left(C+x_{i}\right)=\left(C\cap\left(C+x_{0}\right)\right)+\sum_{j=0}^{i}c_{k_{j}}$
for some $i\in\mathbb{N}_{0}$ and $\sigma_{x_{i}}\left(j\right)=\pm1$
for all $j\in\mathbb{N}_{0}$.

Let $P_{i}:=\left\{ h\mid\sigma_{x_{i}}\left(h\right)=-1\right\} $
be a subset of $\mathbb{N}$. By assumption, $P_{i}$ is empty iff
$\sigma_{x_{i}}\left(k\right)=1$ for all $k$ and we can choose $y:=x_{i}$. 

Suppose $P_{i}$ is nonempty and let $k_{i+1}\in P_{i}$ be the minimal
element. Thus, $\sigma_{x_{i}}\left(k_{i+1}-1\right)=1$ so that $x_{i,k_{i+1}}\in\Delta-1$
by definition of $\sigma$. Therefore, $C\cap\left(C+t\right)$ consists
of $\mu_{x_{i}}\left(k_{i+1}\right)$ copies of $\frac{1}{n^{k_{i+1}}}\left(C\cap\left(C-1+n^{k_{i+1}}\left(x_{i}-\left\lfloor x_{i}\right\rfloor _{k_{i+1}}\right)\right)\right)$
by Lemma \ref{Sec-3-Lem:Contained_(t-t_k)}. Since $C\cap\left(C-1+z\right)=\left(C\cap\left(C+1-z\right)\right)-\left(1-z\right)$
for any real $z$, then 
\begin{align*}
 & \frac{1}{n^{k_{i+1}}}\left(C\cap\left(C-1+n^{k_{i+1}}\left(x_{i}-\left\lfloor x_{i}\right\rfloor _{k_{i+1}}\right)\right)\right)\\
 & \qquad=\frac{1}{n^{k_{i+1}}}\left(C\cap\left(C+1-n^{k_{i+1}}\left(x_{i}-\left\lfloor x_{i}\right\rfloor _{k_{i+1}}\right)\right)\right)-c_{k_{i+1}}\\
 & \qquad=\frac{1}{n^{k_{i+1}}}\left(C\cap\left(C+n^{k_{i+1}}\left(\sum_{j=1+k_{i+1}}^{\infty}\frac{n-1-x_{ij}}{n^{j}}\right)\right)\right)-c_{k_{i+1}}.
\end{align*}
where $c_{k_{i+1}}:=\frac{1}{n^{k_{i+1}}}\left(1-n^{k_{i+1}}\left(x_{i}-\left\lfloor x_{i}\right\rfloor _{k_{i+1}}\right)\right)$.
Since $0\le n^{k_{i+1}}\left(x_{i}-\left\lfloor x_{i}\right\rfloor _{k_{i+1}}\right)\le1$,
then $0\le c_{k_{i+1}}\le\frac{1}{n^{k_{i+1}}}$. Choose $x_{i+1}$
such that 
\[
x_{\left(i+1\right)j}=\begin{cases}
x_{ij} & \text{ for }1\le j\le k_{i+1}-1\\
x_{ij}+1 & \text{ for }j=k_{i+1}\\
n-1-x_{ij} & \text{ for }j>k_{i+1}.
\end{cases}
\]
Thus, $\sigma_{x_{i+1}}\left(j\right)=\sigma_{x_{i}}\left(j\right)=1$
for all $1\le j<k_{i+1}$. It is by definition of $\sigma$ that $\sigma_{x_{i+1}}\left(k_{i+1}\right)=1$
since $x_{\left(i+1\right)k_{i+1}}\in\Delta$. Also, $\sigma_{x_{i+1}}\left(j\right)=\pm1$
for any $j>k_{i+1}$ since $\sigma_{x_{i}}\left(j\right)=\pm1$ by
assumption and
\begin{alignat*}{2}
x_{\left(i+1\right)j}\in & \,\Delta & \text{ iff }x_{ij}\in & \, n-\Delta-1\\
x_{\left(i+1\right)j}\in & \,\Delta-1 & \text{ iff }x_{ij}\in & \, n-\Delta\\
x_{\left(i+1\right)j}\in & \, n-\Delta & \text{ iff }x_{ij}\in & \,\Delta-1\\
x_{\left(i+1\right)j}\in & \, n-\Delta-1 & \text{ iff }x_{ij}\in & \,\Delta.
\end{alignat*}

In particular, $\sigma_{x_{i+1}}\left(j\right)=-\sigma_{x_{i}}\left(j\right)$
for any $j\ge k_{i+1}$. Therefore, $C\cap\left(C+x_{i+1}\right)$
is not empty by Lemma \ref{Sec-2-Lem:Nothing_Goes_Away} so that $x_{i+1}\in F^{+}$.
Since each potential interval $J\subset C_{k_{i+1}}\cap\left(C_{k_{i+1}}+\left\lfloor x_{i}\right\rfloor _{k_{i+1}}\right)$
is an interval case in $C_{k_{i+1}}+\left\lfloor x_{i+1}\right\rfloor _{k_{i+1}}$
by Lemma \ref{Sec-3-Lem:Contained_(t-t_k)} then $C\cap\left(C+x_{i}\right)=\left(C\cap\left(C+x_{i+1}\right)\right)-c_{k_{i+1}}$.
Hence,
\[
C\cap\left(C+x_{i+1}\right)=\left(C\cap\left(C+x_{0}\right)\right)+\sum_{j=0}^{i+1}c_{k_{j}}.
\]
By induction, $\left\{ x_{i}\right\} $ is a sequence in $F^{+}$
such that $C\cap\left(C+x_{i}\right)=\left(C\cap\left(C+x_{0}\right)\right)+\sum_{j=0}^{i}c_{k_{j}}$
and $\sigma_{x_{i}}\left(j\right)=\pm1$ for all $i,j\in\mathbb{N}_{0}$.

By construction, $0\le c_{k_{i}}\le\frac{1}{n^{k_{i}}}$ and $c_{k_{0}}=0$
so that $\sum_{j=0}^{i}c_{k_{j}}\le\sum_{j=1}^{i}\frac{1}{n^{k_{j}}}\le\sum_{j=1}^{k_{i}}\frac{1}{n^{j}}\le\frac{1}{n-1}$
for all $i\in\mathbb{N}_{0}$. Since the sequence $\left\{ \sum_{j=0}^{i}c_{k_{j}}\right\} $
is increasing and bounded above, let $c:=\lim_{j\to\infty}\left\{ \sum_{j=0}^{i}c_{k_{j}}\right\} $.

Let $\varepsilon>0$ be given. Choose $N\in\mathbb{N}$ such that
$\varepsilon>\left(\frac{1}{n}\right)^{k_{N}}>0$ and let $N\le i<j$.
Since $x_{i}$ and $x_{j}$ have been constructed so that the first
$k_{i}$ digits are equal, then $\left|x_{i}-x_{j}\right|\le\left(\frac{1}{n}\right)^{k_{i}}<\varepsilon$.
Therefore, $\left\{ x_{i}\right\} $ is a Cauchy sequence of $F^{+}$
and $y:=\lim_{i\to\infty}\left(x_{i}\right)$ is also in $F^{+}$.
By construction, $y=0._{n}y_{1}y_{2}\ldots$ is the unique value such
that
\begin{equation}
y_{j}=\begin{cases}
t_{j} & \text{ if }\sigma_{t}\left(j-1\right)=1\mbox{\text{ and }}\sigma_{t}\left(j\right)=1\\
t_{j}+1 & \text{ if }\sigma_{t}\left(j-1\right)=1\text{ and }\sigma_{t}\left(j\right)=-1\\
n-1-t_{j} & \text{ if }\sigma_{t}\left(j-1\right)=-1\text{ and }\sigma_{t}\left(j\right)=-1\\
n-t_{j} & \text{ if }\sigma_{t}\left(j-1\right)=-1\mbox{\text{ and }}\sigma_{t}\left(j\right)=1.
\end{cases}\label{Sec-3-Eq:Psi defined}
\end{equation}

Hence, the sequence $\left\{ x_{i}\right\} $ satisfies the conditions
of Lemma \ref{Sec-3-Lem:Sequence of Sets} so that 
\[
C\cap\left(C+t\right)=\left(C\cap\left(C+y\right)\right)-c.
\]
Furthermore, for each $i\in\mathbb{N}$ there exists $k_{j}>i$ such
that $y$ shares the first $k_{j}$ digits of $x_{j}$ and $\sigma_{y}\left(h\right)=\sigma_{x_{j}}\left(h\right)=1$
for all $0\le h\le k_{j}$. Thus $\sigma_{y}\left(h\right)=1$ for
all $h\in\mathbb{N}$.

It remains to show that $y$ is rational whenever $t$ is rational.
Suppose $k\ge p>0$ and $t=0._{n}t_{1}\cdots t_{k-p}\overline{t_{k-p+1}\cdots t_{k}}$.
Let $q$ denote a period of $\sigma_{t}\left(k\right)$ by Lemma \ref{Sec-3-Lem:Sigma_period_q}.
Since $t_{j+q}=t_{j}$ and $\sigma_{t}\left(j+q\right)=\sigma_{t}\left(j\right)$
for any $j>k$, then $y_{j}=y_{j+q}$ by equation (\ref{Sec-3-Eq:Psi defined})
so that $y$ has period $q$.\end{proof}
\begin{rem}
Define the function $\psi:\left[0,1\right]\to\left[0,1\right]$ according
to equation (\ref{Sec-3-Eq:Psi defined}) so that $\psi\left(t\right)=y$.
For example, if $D=\left\{ 0,2,7,9\right\} $, $n=10$, and $t=0._{n}54\overline{4728}$,
then $\psi\left(t\right)=0._{n}55\overline{5272}$. In example \ref{Sec-4-Ex:Irrational_Self_Similar},
we choose $D$, $n$, and $t,t'\in F$ such that $C\left(t\right)=C\left(t'\right)$,
yet $\psi\left(t\right)\neq\psi\left(t'\right)$.
\end{rem}

\subsection{Proof of Theorem \ref{Sec-1-Thm:Structure-when-t-is-rational}.}

We have now developed the machinery necessary to prove the first half
of Theorem \ref{Sec-1-Thm:Rational_tau}. In fact, Theorem \ref{Sec-1-Thm:Structure-when-t-is-rational}
is a special case of the following result. 
\begin{thm}
\label{Sec-3-Thm:Rational_tau}Let $C_{n,D}$ be given and $z\in F$
be arbitrary. Suppose there exists $t\in F^{+}$ such that $C\left(z\right)=C\left(t\right)$,
$t=0._{n}t_{1}\cdots t_{k-p}\overline{t_{k-p+1}\cdots t_{k}}$ for
some period $p$ and integer $k\ge p$, and $\sigma_{t}\left(j\right)=\pm1$
for all $j\in\mathbb{N}_{0}$. If $q$ denotes a period of $\sigma_{t}\left(j\right)$
then there exists a digits set $E=\left\{ 0\le e_{1}<e_{2}<\cdots<e_{r}<n^{q}\right\} $
and corresponding deleted digits Cantor set $B=C_{n^{q},E}$ such
that $C\cap\left(C+t\right)$ consists of $\mu_{t}\left(k\right)$
disjoint copies of $\frac{1}{n^{k}}B$. If $D$ is sparse then $E$
is also sparse.\end{thm}
\begin{proof}
Let $y:=\psi\left(t\right)\in F^{+}$ according to Lemma \ref{Sec-3-Lem:Constructing y Equiv t With Sigma_y=00003D1}
so that $y=0._{n}y_{1}\cdots y_{k}\overline{x_{1}x_{2}\cdots x_{q}}$
does not admit finite $n$-ary representation, $\sigma_{y}\left(j\right)=1$
for all $j$, and $C\cap\left(C+t\right)=\left(C\cap\left(C+y\right)\right)+c$
for some $c\in\mathbb{R}$. Define $x:=n^{k}\left(y-\left\lfloor y\right\rfloor _{k}\right)=0._{n}\overline{x_{1}\cdots x_{q}}$
so that $C\cap\left(C+y\right)$ consists of $\mu_{y}\left(k\right)$
disjoint copies of $\frac{1}{n^{k}}\left(C\cap\left(C+x\right)\right)$
by Lemma \ref{Sec-3-Lem:Contained_(t-t_k)}. We will construct $E$
and show that $C\cap\left(C+x\right)=B$. 

Let $\left\{ S_{d}\right\} _{d\in D}$ be the similarity mappings
which generate $C$. Let $S^{1}\left(a\right):=\bigcup_{d\in D}S_{d}\left(a\right)$
so that $C=S^{1}\left(C\right)$ by definition and let $S^{j}\left(a\right)=\left(S^{j-1}\circ S^{1}\right)\left(a\right)$
for all $j\in\mathbb{N}$. Thus, $C\cap\left(C+x\right)=S^{q}\left(C\right)\cap\left(S^{q}\left(C\right)+x\right)$.
For each $\boldsymbol{u}=\left(u_{1},u_{2},\cdots,u_{q}\right)\in D^{q}$,
define 
\[
S_{\boldsymbol{u}}\left(a\right):=\left(S_{u_{q}}\circ\cdots\circ S_{u_{1}}\right)\left(a\right)=\frac{1}{n^{q}}\left(a+\sum_{j=1}^{q}u_{j}\cdot n^{q-j}\right).
\]

Hence, $C=S^{q}\left(C\right)=\bigcup_{\boldsymbol{u}\in D^{q}}S_{\boldsymbol{u}}\left(C\right)$.
Since $x=\frac{1}{n^{q}}\left(\sum_{j=1}^{q}x_{j}\cdot n^{q-j}\right)+\frac{1}{n^{q}}x$,
let $\boldsymbol{w}:=\left(x_{1},x_{2},\cdots,x_{q}\right)$. Then
for any $\boldsymbol{v}=\boldsymbol{u}+\boldsymbol{w}$, 
\begin{align*}
S_{\boldsymbol{u}}\left(C\right)+x & =S_{\boldsymbol{u}}\left(C\right)+\frac{1}{n^{q}}\left(\sum_{j=1}^{q}x_{j}\cdot n^{q-j}\right)+\frac{x}{n^{q}}\\
 & =\frac{1}{n^{q}}\left(C+x+\sum_{j=1}^{q}\left(u_{j}+x_{j}\right)\cdot n^{q-j}\right)\\
 & =S_{\boldsymbol{v}}\left(C+x\right).
\end{align*}

Therefore, $S_{\boldsymbol{u}}\left(C\right)+x=S_{\boldsymbol{v}}\left(C\right)+\frac{x}{n^{q}}=S_{\boldsymbol{v}}\left(C+x\right)$.
Since $x_{k}\in\Delta$ for all $k$, then $\boldsymbol{v}\in D^{q}$
and
\begin{align*}
C\cap\left(C+x\right) & =\left[\bigcup_{\boldsymbol{u}\in D^{q}}S_{\boldsymbol{u}}\left(C\right)\right]\bigcap\left[\bigcup_{\boldsymbol{v}\in D^{q}+\boldsymbol{w}}\left(S_{\boldsymbol{v}}\left(C+x\right)\right)\right]\\
 & =\bigcup_{\boldsymbol{u}\in D^{q}\cap\left(D^{q}+\boldsymbol{w}\right)}\left(S_{\boldsymbol{u}}\left(C\right)\cap S_{\boldsymbol{u}}\left(C+x\right)\right)\\
 & =\bigcup_{\boldsymbol{u}\in D^{q}\cap\left(D^{q}+\boldsymbol{w}\right)}S_{\boldsymbol{u}}\left(C\cap\left(C+x\right)\right).
\end{align*}

Let $E:=\left\{ \sum_{j=1}^{q}u_{j}\cdot n^{q-j}\mid\boldsymbol{u}\in D^{q}\cap\left(D^{q}+\boldsymbol{w}\right)\right\} $.
Then $B=C_{n^{q},E}$ is the unique, nonempty compact set invariant
under the mapping 
\[
\bigcup_{e\in E}\left(\frac{1}{n^{q}}\left(\cdot+e\right)\right)=\bigcup_{\boldsymbol{u}\in D^{q}\cap\left(D^{q}+\boldsymbol{w}\right)}S_{\boldsymbol{u}}\left(\cdot\right).
\]

Hence, $C\cap\left(C+x\right)=B$. Note that $\#E=\prod_{j=1}^{q}\#D\cap\left(D+x_{j}\right)=\mu_{x}\left(q\right)$.

Suppose $D$ is sparse. It is sufficient to show that $\gamma-\gamma'\ge2$
for any $\gamma\neq\gamma'$ in $\Gamma:=\left\{ \sum_{j=1}^{q}u_{j}\cdot n^{q-j}\mid\boldsymbol{u}\in D^{q}-D^{q}\right\} $
since $E-E\subseteq\Gamma$. Let $\gamma\neq\gamma'$ be arbitrary
and $i$ be the smallest index $1\le i\le q$ such that $\gamma_{i}\neq\gamma'_{i}$.
Without loss of generality, assume $\gamma_{i}>\gamma'_{i}$. Since
$D$ is sparse and $\gamma_{j},\gamma'_{j}\in\Delta$, then $\left|\gamma_{j}-\gamma'_{j}\right|\ge2$
for all $1\le j\le q$. Thus, if $i=q$ then $\left|\gamma-\gamma'\right|=\left|\gamma_{q}-\gamma'_{q}\right|\ge2$.
Otherwise, if $i<q$ then 
\begin{align*}
\left|\gamma-\gamma'\right| & =\left|\sum_{j=i}^{q}\gamma_{j}\cdot n^{q-j}-\sum_{j=i}^{q}\gamma'_{j}\cdot n^{q-j}\right|=\left|\left(\gamma_{i}-\gamma'_{i}\right)n^{q-i}+\sum_{j=i+1}^{q}\left(\gamma_{j}-\gamma'_{j}\right)\cdot n^{q-j}\right|\\
 & \ge\left|2n^{q-i}-\sum_{j=i+1}^{q}\left(n-1\right)\cdot n^{q-j}\right|\ge\left|2n^{q-i}-n^{q-i}\right|\ge n.
\end{align*}

Therefore, $E$ is sparse.
\end{proof}
Theorem \ref{Sec-3-Thm:Rational_tau} shows that any sparse set $C$
and rational $t\in F$ is the finite, disjoint union of self-similar
sets and proves the first half of Theorem \ref{Sec-1-Thm:Rational_tau}.

\subsection{Hausdorff Measure of $C\cap\left(C+t\right)$.}

The structure of the set $C\cap\left(C+t\right)$ is given by Theorem
\ref{Sec-3-Thm:Rational_tau} when $t$ is translate equivalent to
a rational, and Lemma \ref{Sec-3-Lem:Contained_(t-t_k)} when $\sigma_{t}\left(k\right)=\pm1$
for all $k$. This additional structure allows us to apply various
methods for calculating the Hausdorff dimension and measure of $C\cap\left(C+t\right)$.
If $t$ is an arbitrary element of $F$ such that $\sigma_{t}\left(k\right)=\pm1$
for all $k$, the Hausdorff dimension of $C\cap\left(C+t\right)$
can be calculated by methods in \cite{PePh11b} and \cite{PePh11a}.
Specifically, if $t=0._{n}t_{1}t_{2}\cdots t_{k}\overline{t_{k+1}\cdots t_{k+q}}$
and $\sigma_{t}\left(k\right)=1$ for all $k$, then the Hausdorff
dimension of $C\left(t\right)$ is $\frac{1}{q}\sum_{j=1}^{q}\log_{n}\#\left(D\cap\left(D+t_{k+j}\right)\right)$.
\begin{rem}
Assume the notation from Theorem \ref{Sec-3-Thm:Rational_tau}. Furthermore,
suppose $D$ is sparse, then $s:=\log_{n^{q}}\left(\#E\right)$ is
the Hausdorff dimension of $B$ and 
\[
\mathscr{H}^{s}\left(C\cap\left(C+t\right)\right)=\left(\mu_{t}\left(k\right)\right)^{s}\cdot\mathscr{H}^{s}\left(B\right).
\]
 An algorithm for calculating the Hausdorff measure of $B$ in a finite
number of steps is known, see \cite{AySt99}, \cite{Ma86}, and \cite{Ma87}.
An estimate of the $\dim_{H}\left(C\cap\left(C+t\right)\right)$-dimensional
Hausdorff measure of $C\cap\left(C+t\right)$ is given in \cite{PePh11b}
even when the set is is not a finite union of self-similar sets. 
\end{rem}
Proposition \ref{Sec-3-Prop:Hausdorff Measure when D=00003D{0,d}}
gives a formula for the Hausdorff measure of a deleted digits Cantor
set when $D$ contains only two digits.
\begin{prop}
\label{Sec-3-Prop:Hausdorff Measure when D=00003D{0,d}}Let $n\ge3$
and $0\le a<b<n$ be nonnegative integers. If $D=\left\{ a,b\right\} $
and $s:=\log_{n}\left(2\right)$, then $\mathscr{H}^{s}\left(C_{n,D}\right)=\left(\frac{b-a}{n-1}\right)^{s}$.\end{prop}
\begin{proof}
Let $n\ge3$ be given. We may assume $D=\left\{ 0,d\right\} $ where
$d:=b-a\ge1$ and $\Delta=\left\{ -d,0,d\right\} $. If $d\ge2$ then
$D$ is sparse.

Since $\frac{n-1}{d}\cdot D=\left\{ 0,n-1\right\} $ then $B=C_{n,\frac{n-1}{d}\cdot D}$
is the self-similar Cantor set generated by removing the open ``middle''
interval of length $1-2\cdot\frac{1}{n}$. This set is well known
to have measure $\mathscr{H}^{s}\left(B\right)=1$, see e.g., \cite{Hau19},
\cite{Hut81}, \cite{Fal85}, or \cite{Ma86}. Hence, $\frac{d}{n-1}B=\left\{ \sum_{j=1}^{\infty}\frac{t_{j}}{n^{j}}\mid t_{j}\in D\right\} =C$
and $\mathscr{H}^{s}\left(C\right)=\mathscr{H}^{s}\left(\frac{d}{n-1}\cdot B\right)=\left(\frac{d}{n-1}\right)^{s}$.\end{proof}
\begin{example}
\label{Sec-3-Ex:MTC and t=00003D1/4}Let $C=C_{3,\left\{ 0,2\right\} }$
denote the middle thirds Cantor set. Let $t=0._{3}\overline{20}=\frac{3}{4}$
so that $q=2$ is a period of $\sigma_{t}\left(k\right)$. By Theorem
\ref{Sec-3-Thm:Rational_tau}, $C\cap\left(C+t\right)$ is the self-similar
set $C_{9,\left\{ 6,8\right\} }$. If $s:=\log_{9}\left(2\right)$,
then $\mathscr{H}^{s}\left(C\cap\left(C+t\right)\right)=4^{-s}$ by
Proposition \ref{Sec-3-Prop:Hausdorff Measure when D=00003D{0,d}}.
\end{example}

\subsection{\label{Sub-3.4:Non-sparse D}Non-sparse digit sets.}

Many of the results in Section \ref{Sec-3:Rational t} only require
that $t\in F$ satisfy $\sigma_{t}\left(k\right)=\pm1$ for all $k\in\mathbb{N}_{0}$.
In this section we construct specific examples to apply these results
when $D$ is not a sparse digits set. Example \ref{Sec-3-Ex:non-sparse D}
constructs a family of values $t\in F$ when $D$ is not sparse.
\begin{example}
\label{Sec-3-Ex:non-sparse D}Let $n=10$, $D=\left\{ 0,1,2,6,8\right\} $,
and $C=C_{n,D}$. Then $D$ is not sparse, yet $\left\{ 2,8\right\} \subset\Delta\setminus\left(\Delta-1\right)$
where $\setminus$ denotes set subtraction. Thus, any $t\in C_{n,\left\{ 2,8\right\} }$
is such that $\sigma_{t}\left(k\right)=1$ for all $k$ by definition
of $\sigma$. Let $t=0._{10}\overline{2}=\frac{2}{9}$. Since $D\cap\left(D+2\right)=\left\{ 2,8\right\} $
then $C\cap\left(C+t\right)=\left\{ 0._{n}x_{1}x_{2}\ldots\mid x_{k}\in\left\{ 2,8\right\} \right\} =C_{n,\left\{ 2,8\right\} }$.
If $s:=\log_{n}\left(2\right)$ then $\mathscr{H}^{s}\left(C\cap\left(C+t\right)\right)=\left(\frac{2}{3}\right)^{s}$
by Proposition \ref{Sec-3-Prop:Hausdorff Measure when D=00003D{0,d}}.
\end{example}
In specific cases, these methods can be applied to analyze values
$t\in F$ when $D$ is not sparse and $\sigma_{t}\left(k\right)\neq\pm1$
for some $k$. 
\begin{example}
\label{Sec-3-Ex:Self-Similar with sigma=00003Di}Let $D=\left\{ 0,2,4,7,\ldots,4+3r\right\} $
for some $r>2$ and $n>4+3\left(r+1\right)$. Choose $t=0._{n}\overline{2}\in F$.
Note that $D$ is not sparse and $\sigma_{t}\left(k\right)=i$ for
all $k\ge1$. For each $k$, $C_{k}\cap\left(C_{k}+\left\lfloor t\right\rfloor _{k}\right)$
contains $2^{k}$ interval cases and $r\cdot2^{k-1}$ potential interval
cases, however the potential interval cases never contain points in
$C_{n,D}\cap\left(C_{n,D}+t\right)$ since $2$ is neither in $n-\Delta$
nor $n-\Delta-1$. 

For each $k$, let $I_{k}$ denote the collection of $2^{k}$ interval
cases of $C_{k}\cap\left(C_{k}+\left\lfloor t\right\rfloor _{k}\right)$
so that $C_{n,D}\cap\left(C_{n,D}+t\right)\subset\bigcup_{J\in I_{k}}J$
for each $k$. If $E:=\left\{ 0,2,4\right\} $, then $I_{k}$ consists
of the same $2^{k}$ intervals chosen from the $k^{\text{th}}$ step
in the construction of $C_{n,E}\cap\left(C_{n,E}+t\right)$. Since
$C_{n,E}\cap\left(C_{n,E}+t\right)=\bigcap_{k=1}^{\infty}\left(\bigcup_{J\in I_{k}}J\right)$
implies $C_{n,D}\cap\left(C_{n,D}+t\right)\subseteq C_{n,E}\cap\left(C_{n,E}+t\right)$,
and $E\subset D$, then 
\[
C_{n,D}\cap\left(C_{n,D}+t\right)=C_{n,E}\cap\left(C_{n,E}+t\right).
\]

Since $E$ is sparse, then $C_{n,E}\cap\left(C_{n,E}+t\right)=C_{n,\left\{ 2,4\right\} }$
and $\mathscr{H}^{s}\left(C_{n,\left\{ 2,4\right\} }\right)=\left(\frac{2}{n-1}\right)^{s}$
when $s:=\log_{n}\left(2\right)$ by Proposition \ref{Sec-3-Prop:Hausdorff Measure when D=00003D{0,d}}.
Thus, for a specific choice of $t$, $D$, and $n$, we can apply
our method even though $D$ is not sparse and $\sigma_{t}\left(k\right)$
does not equal $\pm1$ for any $k$.
\end{example}

\section{\label{Sec-4:Unions-of-Self-Similar}Unions of Self-Similar Sets}

In this section we prove the second half of Theorem \ref{Sec-1-Thm:Rational_tau}. 
\begin{rem}
A real number $\alpha\in\left[0,1\right]$ has a \emph{$\Delta$ representation,}
if $\alpha=\sum_{k=1}^{\infty}\frac{\alpha_{k}}{n^{k}}=0._{n}\alpha_{1}\alpha_{2}\ldots$
and each $\alpha_{k}$ is a digit of $\Delta$ for all $k$. Let $\left\lfloor \alpha\right\rfloor _{k}=0._{n}\alpha_{1}\alpha_{2}\ldots\alpha_{k}.$
Note $\Delta$ representations allows the digits $\alpha_{k}$ to
be positive for some $k$ and negative for other $k.$ It is easy
to see that $F$ is the self-similar set $\left\{ 0._{n}\alpha_{1}\alpha_{2}\cdots\mid\alpha_{k}\in\Delta\right\} $,
see e.g., \cite{PePh11a}. Throughout this paper we will denote $\Delta$
representations as $\alpha$, $\gamma$ and reserve $t$, $x$, and
$y$ for $n$-ary representations.
\end{rem}
The discussion in Section \ref{Sec-2:Old-Stuff} holds for $C\cap\left(C+\alpha\right)$
with $\Delta$ representations of $\alpha$ in $F.$ Detailed proofs
are in \cite{PePh11b} for the case of $n$-ary representations. The
key observations are that $\Delta=-\Delta$, $\left|h\right|\in\Delta-1$
iff $-\left|h\right|\in\Delta+1$, $\sigma_{\alpha}(k)$ only has
values in $\{1,i\},$ and the geometric configurations we called potentially
empty cases in Section \ref{Sec-2:Old-Stuff} are now potential interval
cases. 
\begin{rem}
Suppose $\alpha\in F$ admits finite $\Delta$ representation and
let $j,k\in\mathbb{Z}$ satisfy $\alpha=\left\lfloor \alpha\right\rfloor _{k}=\frac{j}{n^{k}}$.
We may define an $n$-ary representation $t:=0._{n}t_{1}t_{2}\cdots t_{k}=\frac{\left|j\right|}{n^{k}}$
so that $C\left(t\right)=C\left(\left|\alpha\right|\right)=C\left(\alpha\right)$
by definition of $F$. Thus, we may apply Theorem 3.1 of \cite{PePh11b}
so that $C\cap\left(C+\alpha\right)=A\cup B$, where $A$ and $B$
are the sets defined in Remark \ref{Sec-2-Rem:Finite n-ary Representations}.
Hence, we will focus our analysis on values $\alpha\in F$ which do
not admit finite $\Delta$ representations.
\end{rem}

\subsection{\label{Sub-4.1:Delta Translation Equivalence}Translation Equivalence
of $\Delta$ representations.}

The next result continues our investigation of translation equivalence.
More precisely, we describe translation equivalence in terms of the
digit set $D.$ 
\begin{thm}
\label{Sec-4-Thm:Translate using integer sequence} Let $n\geq3$
and let $D$ be a sparse digit set. Let $\alpha=\sum_{k=1}^{\infty}\alpha_{k}n^{-k},$
$\beta=\sum_{k=1}^{\infty}\beta_{k}n^{-k},$ and $\delta=\sum_{k=1}^{\infty}\delta_{k}n^{-k}.$
If $D\cap\left(D+\alpha_{k}\right)=D\cap\left(D+\beta_{k}\right)+\delta_{k}$
for all $k\geq1,$ then $C\cap\left(C+\alpha\right)=\left(C\cap\left(C+\beta\right)\right)+\delta.$ \end{thm}
\begin{proof}
Recall, if $A$ is a set of real numbers and $t$ is a real number
then $tA=\left\{ ta\mid a\in A\right\} $ and if $A$ and $B$ are
two sets of real numbers then $A+B=\left\{ a+b\mid a\in A,b\in B\right\} .$ 

Let $C_{0}=[0,1].$ The refinement is $C_{1}=\frac{1}{n}\left(D+C_{0}\right)=\frac{1}{n}D+\left[0,1/n\right].$
The refinement of $C_{1}$ is $C_{2}=\frac{1}{n}\left(D+C_{1}\right)=\frac{1}{n}D+\frac{1}{n^{2}}D+[0,1/n^{2}].$
Continuing in this manner we see that 
\[
C_{k}=\frac{1}{n}D+\frac{1}{n^{2}}D+\cdots+\frac{1}{n^{k}}D+\left[0,1/n^{k}\right]=\sum_{i=1}^{k}\frac{1}{n^{i}}D+\left[0,1/n^{k}\right].
\]
By sparsity of $D$ the distance between any two of these intervals
is $\ell/n^{k}$ for some integer $1\leq\ell<n^{k}.$ Therefore,
\[
C_{k}+\sum_{i=1}^{k}\alpha_{i}n^{-i}=\sum_{i=1}^{k}\frac{1}{n^{i}}\left(D+\alpha_{i}\right)+\left[0,1/n^{k}\right]
\]
 and consequently, 
\[
C_{k}\cap\left(C_{k}+\sum_{i=1}^{k}\alpha_{i}n^{-i}\right)=\sum_{i=1}^{k}\frac{1}{n^{i}}\left(D\cap\left(D+\alpha_{i}\right)\right)+\left[0,1/n^{k}\right].
\]
Using $D\cap\left(D+\alpha_{i}\right)=\left(D\cap\left(D+\beta_{i}\right)\right)+\delta_{i}$
for $1\leq i\leq k,$ it follows that 
\[
C_{k}\cap\left(C_{k}+\sum_{i=1}^{k}\alpha_{i}n^{-i}\right)=\left(C_{k}\cap\left(C_{k}+\sum_{i=1}^{k}\beta_{i}n^{-i}\right)\right)+\sum_{i=1}^{k}\delta_{i}n^{-i}
\]
 and that this is a collection of intervals each of length $1/n^{k}.$
Let $\beta^{(0)}=\beta,$ and $\beta^{(k)}=\sum_{i=1}^{k}\alpha_{i}n^{-i}+\sum_{i=k+1}^{\infty}\beta_{i}n^{-i}.$
Using $\beta^{(k)}-\sum_{i=1}^{k}\alpha_{i}n^{-i}=\beta^{(0)}-\sum_{i=1}^{k}\beta_{i}n^{-i}$
we conclude
\[
C_{k}\cap\left(C_{k}+\beta^{(k)}\right)=\left(C_{k}\cap\left(C_{k}+\beta^{(0)}\right)\right)+\sum_{i=1}^{k}\delta_{i}n^{-i}
\]
is a collection of intervals each of length $\frac{1}{n^{k}}\left(1-\left|\sum_{i=k+1}^{\infty}\beta_{i}n^{-i}\right|\right).$
Repeatedly refining the intervals in $C_{k}$ we get 
\[
C_{j}\cap\left(C_{j}+\beta^{(k)}\right)=\left(C_{j}\cap\left(C_{j}+\beta^{(0)}\right)\right)+\sum_{i=1}^{k}\delta_{i}n^{-i}
\]
for $j\geq k.$ Consequently, 
\[
C\cap\left(C+\beta^{(k)}\right)=\left(C\cap\left(C+\beta^{(0)}\right)\right)+\sum_{i=1}^{k}\delta_{i}n^{-i}.
\]
Since $\beta^{(k)}\to\alpha$ as $k\to\infty$ the result follows
from Lemma \ref{Sec-3-Lem:Sequence of Sets}. 
\end{proof}
Example \ref{Sec-4-Ex:Irrational_Self_Similar} applies Theorem \ref{Sec-4-Thm:Translate using integer sequence}
to construct an uncountable set of values $x\in F^{+}$ which are
not only translation equivalent, but all generate the same set $C\cap\left(C+t\right)$.
\begin{example}
\label{Sec-4-Ex:Irrational_Self_Similar}Let $D=\left\{ 0,5,7\right\} $
and $n=8$ so that $\Delta=\left\{ -7,-5,-2,0,2,5,7\right\} $ and
$C=C_{n,D}$ is sparse. Choose $t:=0._{8}\overline{07}$. Then $C\cap\left(C+t\right)=C_{64,\left\{ 7,47,63\right\} }$
has dimension $s:=\log_{64}\left(3\right)$ and measure $0<\mathscr{H}^{s}\left(C\cap\left(C+t\right)\right)<\infty$.

Note that $D\cap\left(D+2\right)=\left\{ 7\right\} $, $D\cap\left(D+5\right)=\left\{ 5\right\} $,
and $D\cap\left(D+7\right)=\left\{ 7\right\} $. Let $x\in\left[0,1\right]$
with ternary representation $x:=0._{3}x_{1}x_{2}\ldots$. Let $f:\left[0,1\right]\to F$
such that $f\left(x\right)=0._{8}y_{1}y_{2}\ldots$ consists of digits
\begin{align*}
y_{2k-1} & :=0\\
y_{2k} & :=\begin{cases}
2 & \text{ if }x_{k}=0\\
5 & \text{ if }x_{k}=1\\
7 & \text{ if }x_{k}=2.
\end{cases}
\end{align*}
Then, $f\left(1\right)=t$, $\sigma_{f\left(x\right)}\left(k\right)=1$
for all $k$, and $f\left(x\right)$ is irrational whenever $x$ is
irrational. It follows from Theorem \ref{Sec-4-Thm:Translate using integer sequence}
that $\left(C\cap\left(C+f\left(x\right)\right)\right)+c=C\cap\left(C+t\right)$
is self-similar, in particular, $f\left(x\right)$ is translation
equivalent to $t$ for all $x\in\left[0,1\right]$. It is, perhaps,
interesting to note that since $D\cap\left(D+2\right)=D\cap\left(D+7\right)$
then $C\cap\left(C+f\left(x\right)\right)=C\cap\left(C+t\right)$
for any representation of $x$ chosen from the middle thirds Cantor
set $C_{3,\left\{ 0,2\right\} }$.
\end{example}
Define $\Delta^{+}:=\Delta\cap\left[0,\infty\right)$. We say that
$\alpha\in F$ has a\emph{ $\Delta^{+}$ representation} if $\alpha=0._{n}\alpha_{1}\alpha_{2}\ldots$
such that each $\alpha_{k}\in\Delta^{+}$ for all $k$. According
to Corollary \ref{Sec-4-Cor:Positive Delta representations.}, when
$D$ is sparse, we can restrict our analysis to $\Delta^{+}$ representations
in $F$ withoutout loss of generality.
\begin{cor}
\label{Sec-4-Cor:Positive Delta representations.}Suppose $D$ is
sparse. If $\alpha\in F^{+}$ has $\Delta$ representation $\alpha=\sum_{k=1}^{\infty}\alpha_{k}n^{-k},$
then $\alpha$ is translation equivalent to $\widetilde{\alpha}:=\sum_{k=1}^{\infty}\left|\alpha_{k}\right|n^{-k}$. \end{cor}
\begin{proof}
Let $\alpha\in F$ be given with $\Delta$ representation $\alpha:=\sum_{k=1}^{\infty}\alpha_{k}n^{-k}$.
If $\alpha_{k}\le0$ for some $k$, then $\alpha_{k},\pm\left|\alpha_{k}\right|$
are integers such that $D\cap\left(D+\alpha_{k}\right)=D\cap\left(D+\left|\alpha_{k}\right|\right)-\left|\alpha_{k}\right|.$
Hence $\alpha$ is translation equivalent to $\widetilde{\alpha}=\sum_{k=1}^{\infty}\left|\alpha_{k}\right|n^{-k}$
by Theorem \ref{Sec-4-Thm:Translate using integer sequence}.\end{proof}
\begin{rem}
For the middle thirds Cantor set this shows that any intersection
$C\cap\left(C+t\right)$ is a translate of an intersection $C\cap\left(C+s\right)$
with $s$ in $C.$
\end{rem}

\subsection{Proof of Theorem \ref{Sec-1-Thm:Rational_tau}.}

In this section, we will prove the second part of Theorem \ref{Sec-1-Thm:Rational_tau}.
We begin by showing that $\Delta^{+}$ representations are unique.
\begin{lem}
\label{Sec-4-Lem:Unique Delta Representations}Let $D$ be sparse.
If $\alpha\in F^{+}$ has a representation $\alpha=\sum_{k=1}^{\infty}\frac{\alpha_{k}}{n^{k}}$
with digits $\alpha_{k}\in\Delta^{+}$ for all $k$, then this representation
is unique. \end{lem}
\begin{proof}
Suppose the $\Delta^{+}$ representation $\alpha=0._{n}\alpha_{1}\alpha_{2}\ldots$
is not unique. Any sequence $\left\{ \alpha_{k}\right\} \subseteq\Delta^{+}$
is also a sequence of $\left\{ 0,1,\ldots,n-1\right\} $ so that $\sum_{k=1}^{\infty}\alpha_{k}n^{-k}$
is an $n$-ary representation of $\alpha$. Thus $\alpha$ has two
$n$-ary representations with digits in $\Delta^{+}$, namely, $0._{n}\alpha_{1}\ldots\alpha_{k}b\overline{0}$
and $0._{n}\alpha_{1}\ldots\alpha_{k}\left(b-1\right)\overline{\left(n-1\right)}$
for some $k$ and $1\le b\le n-1$. Hence, $b$ and $b-1$ are both
elements of $\Delta$, which contradicts that $D$ is sparse by assumption.
Therefore, the $n$-ary representation with digits in $\Delta^{+}$
is unique.
\end{proof}
Since the $\Delta^{+}$ representation is unique whenever $D$ is
sparse, we can classify the set $C\cap\left(C+\alpha\right)$ in terms
of the digits $\left\{ \alpha_{k}\right\} $.
\begin{lem}
\label{Sec-4-Lem:Classify C-cap-(C+a) using digits}Let $D$ be sparse.
If $\alpha\in F$ has $\Delta^{+}$ representation $\alpha=0._{n}\alpha_{1}\alpha_{2}\ldots$,
then 
\[
C\cap\left(C+\alpha\right)=\left\{ 0._{n}x_{1}x_{2}\ldots\mid x_{k}\in D\cap\left(D+\alpha_{k}\right)\right\} .
\]
\end{lem}
\begin{proof}
Let $x\in C$ be arbitrary and choose $y\in C$ such that $x=y+\alpha$.
Denote $x:=0._{n}x_{1}x_{2}\ldots$ and $y:=0._{n}y_{1}y_{2}\ldots$
such that $x_{k},y_{k}\in D$. Suppose $x_{k}\neq y_{k}+\alpha_{k}$
for some $k$. Without loss of generality, suppose $k-1=\min\left\{ j\mid x_{j}\neq y_{j}+\alpha_{j}\right\} $.

If $y_{k}+\alpha_{k}<n$ then $x$ has two $n$-ary representations
with digits strictly contained in $D\subseteq\Delta^{+}$, which is
a contradiction by Lemma \ref{Sec-4-Lem:Unique Delta Representations}. 

If $y_{k}+\alpha_{k}\ge n$ then $0\le y_{k},\alpha_{k}<n$ implies
$0\le y_{k}+\alpha_{k}-n<n$. Thus, $x=0._{n}x_{1}\ldots x_{k-1}x_{k}\ldots=0._{n}x_{1}\ldots\left(x_{k-1}+1\right)\left(y_{k}+\alpha_{k}-n\right)\ldots$
has two different $n$-ary representations. Therefore $x_{j}=n-1\in D$
and $y_{j}+\alpha_{j}-n=0$ for all $j\ge k$. However, $\alpha_{k}\in\Delta$
by definition and $\alpha_{k}-1=\left(n-1\right)-y_{k}$ is some element
of $D-D=\Delta$, which contradicts that $D$ is sparse.

Hence, $x_{k}=y_{k}+\alpha_{k}$ for each $k$ and $C\cap\left(C+\alpha\right)\subseteq\left\{ 0._{n}x_{1}x_{2}\cdots\mid x_{k}\in D\cap\left(D+\alpha_{k}\right)\right\} $
since $x$ is arbitrary. The reverse inclusion follows immediately
since $C=\left\{ 0._{n}x_{1}x_{2}\cdots\mid x_{k}\in D\right\} $.
\end{proof}
When $\alpha$ has $\Delta^{+}$ representation, then $\inf\left(C\cap\left(C+\alpha\right)\right)=\sum_{k=1}^{\infty}n^{-k}\cdot\min\left(D\cap\left(D+\alpha_{k}\right)\right)$
according to Lemma \ref{Sec-4-Lem:Classify C-cap-(C+a) using digits}.
For each $\delta\in\Delta^{+}$, define $D_{\delta}:=D\cap\left(D+\delta\right)-\min\left(D\cap\left(D+\delta\right)\right)$
so that $0\in D_{\delta}\subseteq\Delta^{+}$ and 
\begin{equation}
C(\alpha)=C\cap\left(C+\alpha\right)-\sum_{k=1}^{\infty}\frac{\min\left(D\cap\left(D+\alpha_{k}\right)\right)}{n^{k}}=\left\{ \sum_{k=1}^{\infty}\frac{x_{k}}{n^{k}}\mid x_{k}\in D_{\alpha_{k}}\right\} .\label{Sec-4-Eq:C_alpha defined}
\end{equation}

This leads to the following Corollary to Theorem \ref{Sec-4-Thm:Translate using integer sequence}:
\begin{cor}
\label{Sec-4-Cor:C(a) equals C(g) iff Da_k equals Dg_k}Let $D$ be
sparse and suppose $0._{n}\alpha_{1}\alpha_{2}\ldots$ is a $\Delta^{+}$
representation for $\alpha$. Then $C\left(\alpha\right)=C\left(\gamma\right)$
if and only if $D_{\alpha_{k}}=D_{\gamma_{k}}$ for all $k$.
\end{cor}
We now prove the second part of Theorem \ref{Sec-1-Thm:Rational_tau}
when $C\left(\alpha\right)$ is a finite set.
\begin{lem}
\label{Sec-4-Lem:Finite C_a}Let $D$ be sparse and $\alpha\in F$
given. If $C(\alpha)$ is a finite set then $\alpha$ is equivalent
to a rational. In this case, $C\cap\left(C+\alpha\right)$ is the
finite, disjoint union of trivial self-similar sets.\end{lem}
\begin{proof}
Suppose $C(\alpha)$ is finite and let $K:=\left\{ k\mid\left\{ 0\right\} \varsubsetneqq D_{\alpha_{k}}\right\} $.
Suppose $K=\left\{ k_{1}<k_{2}<\cdots\right\} $ is an infinite subset
of $\mathbb{N}$. For any $x=0._{2}x_{1}x_{2}\ldots$ in $\left[0,1\right]$
with $x_{k}\in\left\{ 0,1\right\} $, define 
\[
f\left(x\right):=\sum_{j=1}^{\infty}\left(\frac{x_{j}}{n^{k_{j}}}\cdot\min\left\{ a>0\mid a\in D_{\alpha_{k_{j}}}\right\} \right).
\]
Thus, $f\left(x\right)\in C(\alpha)$ for all $x\in\left[0,1\right]$
and $f\left(x\right)\neq f\left(y\right)$ for any $x\neq y$ so that
$f\left(\left[0,1\right]\right)$ is an uncountably infinite subset
of $C\left(\alpha\right)$. This contradicts the assumption that $C(\alpha)$
is finite, hence $K$ is either a finite subset of $\mathbb{N}$ or
empty. If $K$ is empty then $D_{\alpha_{k}}=\left\{ 0\right\} $
for all $k$ and $C(\alpha)=\left\{ 0\right\} =C\left(\frac{d_{m}}{n-1}\right)$.

Suppose $K$ is finite. Let $k=\max\left(K\right)$ and define $\gamma:=0._{n}\alpha_{1}\alpha_{2}\cdots\alpha_{k}\overline{d_{m}}$
so that $D\cap\left(D+\alpha_{j}\right)=D\cap\left(D+\gamma_{j}\right)$
for each $j\le k$. Since $D\cap\left(D+\alpha_{j}\right)=\left\{ d_{i_{j}}\right\} $
and $D\cap\left(D+\gamma_{j}\right)=\left\{ d_{m}\right\} $ for $j>k$
then $D_{\alpha_{j}}=\left\{ 0\right\} =D_{\gamma_{j}}$. Hence, $C(\alpha)=C(\gamma)$
so that $\alpha$ is translation equivalent to $\gamma$. 
\end{proof}
The proof of Lemma \ref{Sec-4-Lem:Finite C_a} shows that $C\cap\left(C+\alpha\right)$
is either finite or uncountably infinite. Note that $\alpha$ need
not admit finite $n$-ary representation in the proof of Lemma \ref{Sec-4-Lem:Finite C_a}.
Example \ref{Sec-4-Ex:Finite set with Irrational g} exhibits an irrational
number $\alpha$ such that $C\left(\alpha\right)$ is finite.
\begin{example}
\label{Sec-4-Ex:Finite set with Irrational g}Let $D=\left\{ 0,5,7\right\} $,
$n=8$, and $\alpha=0._{n}0\overline{7}$ so that $C\cap\left(C+\alpha\right)=\left\{ \frac{1}{8},\frac{3}{4},1\right\} $.
By defining $\gamma$ such that $\gamma_{k}=\alpha_{k}$ except $\gamma_{2k}=2$
on a sparse set of $k$'s (larger than $1$) then $C\cap\left(C+\gamma\right)=C\cap\left(C+\alpha\right)$
by Theorem \ref{Sec-4-Thm:Translate using integer sequence} yet $\gamma$
does not admit finite $n$-ary representation.\end{example}
\begin{lem}
\label{Sec-4-Lem:Subset_Equivalence}Let $D$ be sparse and $\alpha\in F$
given. There exist nonnegative integers $k,q$ such that $D_{\alpha_{j}}\subseteq D_{\alpha_{j+q}}$
for all $j>k$ if and only if $\alpha$ is translation equivalent
to a rational number.\end{lem}
\begin{proof}
Suppose $D_{\alpha_{j}}\subseteq D_{\alpha_{j+q}}$ for all $j>k$.
Since $D_{\alpha_{j}}\subset\left\{ 0,1,\ldots,d_{m}\right\} $ for
all $j$, then for each $1\le i\le q$ there exists a chain 
\[
D_{\alpha_{k+i}}\subseteq D_{\alpha_{k+i+q}}\subseteq\cdots\subseteq D_{\alpha_{k+i+jq}}\subseteq\left\{ 0,1,\ldots,d_{m}\right\} .
\]

Therefore equality holds for all $D_{\alpha_{k+i+jq}}$ after a certain
point. For each $i$, let $h_{i}$ be a value such that $D_{\alpha_{k+i+h_{i}q}}=D_{\alpha_{k+i+\left(h_{i}+j\right)q}}$
for all $j\ge0$. If $h:=\max_{i}\left(h_{i}\right)$ then $D_{\alpha_{k+i+hq}}=D_{\alpha_{k+i+\left(h+j\right)q}}$
for all $1\le i\le q$ and $j\ge0$. Let $\gamma:=0._{n}\alpha_{1}\cdots\alpha_{k+hq}\overline{\alpha_{k+hq+1}\cdots\alpha_{k+\left(h+1\right)q}}$.
Then $D_{\alpha_{j}}=D_{\gamma_{j}}$ for all $j\in\mathbb{N}$ so
that $C\left(\alpha\right)=C\left(\gamma\right)$.

Conversely, suppose $\gamma=0._{n}\gamma_{1}\cdots\gamma_{k}\overline{\gamma_{k+1}\cdots\gamma_{k+q}}$
is translation equivalent to $\alpha$. Then $C\left(\alpha\right)=C\left(\gamma\right)$
and $D_{\alpha_{j}}=D_{\gamma_{j}}$for all $j\in\mathbb{N}$ by equation
(\ref{Sec-4-Eq:C_alpha defined}). Thus, $D_{\alpha_{j}}=D_{\gamma_{j}}=D_{\gamma_{j+q}}=D_{\alpha_{j+q}}$
for all $j>k$.
\end{proof}
We now have the tools required to prove the second half of Theorem
\ref{Sec-1-Thm:Rational_tau}.
\begin{thm}
\label{Sec-4-Thm:C_alpha-cap-[0,e] is self-similar}Let $D$ be sparse
and $\alpha\in F$ be given. Suppose there exists $\varepsilon>0$
such that $C(\alpha)\cap\left[0,\varepsilon\right]$ is a self-similar
set generated by similarities $f_{j}(x)=r_{j}x+b_{j}$ where $r_{i}=n^{-q_{i}}$
for some $q_{i}\in\mathbb{Z}$. Then $\alpha$ is translation equivalent
to a rational number.\end{thm}
\begin{proof}
According to Lemma \ref{Sub-4.1:Delta Translation Equivalence} we
may assume $\alpha$ has a $\Delta^{+}$ representation. Let $\varepsilon>0$
be a value such that $C\left(\alpha\right)\cap\left[0,\varepsilon\right]=T$
is a self-similar set. We may assume that $b_{1}<b_{2}<\cdots<b_{\ell}$
so that $b_{1}=0$ and $f_{1}\left(x\right)=x\cdot n^{-q_{1}}$. Choose
$k\in\mathbb{N}$ such that $\varepsilon\ge n^{-k}>0$ and let $j>k$
be arbitrary. Let $d\in D_{\alpha_{j}}$ be arbitrary so that $d\cdot n^{-j}\in T\subset C(\alpha)$
by equation (\ref{Sec-4-Eq:C_alpha defined}). We note that the representation
$d\cdot n^{-j}$ is unique by Lemma \ref{Sec-4-Lem:Unique Delta Representations}.
Thus, $f_{1}\left(d\cdot n^{-j}\right)=d\cdot r_{1}\cdot n^{-j}=d\cdot n^{-\left(j+q_{1}\right)}\in C\left(\alpha\right)$
and $d\in D_{\alpha_{j+q_{1}}}$. Since $j$ and $d$ are arbitrary,
then $D_{\alpha_{j}}\subseteq D_{\alpha_{j+q_{1}}}$ for all $j>k$
and $\alpha$ is translation equivalent to a rational number by Lemma
\ref{Sec-4-Lem:Subset_Equivalence}.
\end{proof}
This completes the proof of Theorem \ref{Sec-1-Thm:Rational_tau}.
Theorem \ref{Sec-4-Thm:C_alpha-cap-[0,e] is self-similar} shows that
if $C(\alpha)\cap\left[0,\varepsilon\right]$ is constructed by specific
similarity mappings, then $\alpha$ is translation equivalent to a
rational number and, by Theorem \ref{Sec-3-Thm:Rational_tau}, can
be expressed as 
\[
C\cap\left(C+\alpha\right)=\bigcup_{j=1}^{N}\left(C_{n^{2p},E}+\eta_{j}\right)
\]
for some $\eta_{1}<\eta_{2}<\cdots<\eta_{N}$.

\subsection{When is $C(\alpha)$ self-similar?}

Essentially, half of the answer to this question is provided by a
calculation on page 307 of \cite{LYZ11} and the other half by an
elaboration on the proof of Theorem \ref{Sec-4-Thm:C_alpha-cap-[0,e] is self-similar}. 

Note that, if $C(\alpha)$ is self-similar, then we may choose $k=0$
in the proof of Theorem \ref{Sec-4-Thm:C_alpha-cap-[0,e] is self-similar}.
Hence, this proof shows that for some $q>0$ we have 
\begin{equation}
D_{\alpha_{j}}\subseteq D_{\alpha_{j+q}}\text{ for all }j>0.\label{eq:sp-uniform}
\end{equation}
But this condition is not sufficient for $C(\alpha)$ to be self-similar.
For this we need the stronger condition that $\alpha$ is \emph{strongly
periodic} in the sense that there exists $\widetilde{D}_{\alpha_{j}}$
such that 
\begin{equation}
D_{\alpha_{j}}+\widetilde{D}_{\alpha_{j}}=D_{\alpha_{j+q}}\text{ for all }j>0.\label{eq:sp-sparse}
\end{equation}
Note that (\ref{eq:sp-uniform}) and (\ref{eq:sp-sparse}) are equivalent,
when $D$ is assumed to be a uniform set. Clearly, whether of not
a given $\alpha$ satisfies (\ref{eq:sp-uniform}) or (\ref{eq:sp-sparse})
depends on the set $D.$ 

The following is a restatement of Theorem \ref{Sec-1-Thm:C(t) self-similar iff }. 
\begin{thm}
\label{Sec-4-thm:self-similar}If $D$ is sparse, then $C\cap(C+\alpha)$
is self-similar generated by a finite set of similarities $f_{j}(x)=n^{-q}x+b_{j}$
with $q\in\mathbb{N}$ if and only if $\alpha$ is strongly periodic. \end{thm}
\begin{proof}
Suppose (\ref{eq:sp-sparse}) holds. Since (\ref{eq:sp-sparse}) implies
(\ref{eq:sp-uniform}) it follows that $D_{\alpha_{j}}=D_{\alpha_{j+q}}$
for all sufficiently large $j.$ Hence, for some $p>0,$ we have 
\[
D_{\alpha_{j}}+\widetilde{D}_{\alpha_{j}}=D_{\alpha_{j+pq}}\text{ when }j\leq pq\text{ and }D_{\alpha_{j}}=D_{\alpha_{j+pq}}\text{ when }j>pq.
\]
Consequently, 
\begin{equation}
D_{\alpha_{j+pqr}}=D_{\alpha_{j}}+\widetilde{D}_{\alpha_{j}},\text{ when }j\leq pq\text{ and }r\geq1.\label{Sec-4-eq:digit-self-sim}
\end{equation}
It follows now from the calculation on the top half of page 307 of
\cite{LYZ11} that, $C(\alpha)$ is a self-similar set. For the convenience
of the reader we sketch the details. Let $x\in C(\alpha).$ Use (\ref{Sec-4-Eq:C_alpha defined})
to write $x=\sum_{k}x_{k}n^{-k},$ with $x_{k}\in D_{\alpha_{k}}.$
By (\ref{Sec-4-eq:digit-self-sim}) we can write
\[
x_{k+pqr}=y_{r,k}+z_{r,k},y_{r,k}\in D_{\alpha_{k}},z_{r,k}\in\widetilde{D}_{\alpha_{k}}\text{ when 1\ensuremath{\leq}k\ensuremath{\leq}pq, 1\ensuremath{\leq}r}
\]
and $y_{0,k}=x_{k}$ when $1\leq k\leq pq.$ Then 
\begin{align*}
\sum_{k}x_{k}n^{-k} & =\sum_{r=0}^{\infty}\sum_{k=1}^{pq}x_{k+pqr}n^{-\left(k+pqr\right)}\\
 & =\sum_{k=1}^{pq}y_{0,k}n^{-k}+\sum_{r=1}^{\infty}n^{-pqr}\sum_{k=1}^{pq}(y_{r,k}+z_{r,k})n^{-k}\\
 & =\sum_{r=0}^{\infty}\left(\sum_{k=1}^{pq}\left(y_{r,k}+z_{r+1}n^{-pq}\right)n^{-k}\right)n^{-pqr}.
\end{align*}
It follows that $C(\alpha)$ is generated by the similarities
\[
f_{b}(x)=n^{-pq}x+b,b\in B,
\]
where $B=\left\{ \left.\sum_{k=1}^{pq}\left(y_{k}+z_{k}n^{-pq}\right)n^{-k}\right|y_{k}\in D_{\alpha_{k}},z_{k}\in\widetilde{D}_{\alpha_{k}}\right\} .$ 

On the other hand, suppose $C(\alpha)$ is generated by the similarities
$f_{j}(x)=n^{-q}x+b_{j},$ $j=1,2,\ldots,L.$ Since $0$ is in $C(\alpha)$
it follows that $b_{j}$ is in $C(\alpha)$ for all $j.$ Write
\[
b_{j}=\sum_{k}b_{j,k}n^{-k},\text{ with }b_{j,k}\in D_{\alpha_{k}}.
\]
Let $\widetilde{D}_{\alpha_{k}}=\left\{ b_{j,k+q}\mid j=1,2,\ldots,L\right\} .$
For any $x=\sum_{k}x_{k}n^{-k},$ $x_{k}\in D_{\alpha_{k}}$ we have
\[
f_{j}(x)=\sum_{k=1}^{q}b_{j,k}n^{-k}+\sum_{k=1}^{\infty}\left(b_{j,k+q}+x_{k}\right)n^{-\left(k+q\right)}.
\]
Since $D$ is sparse it follows that $b_{j,k+q}+x_{k}$ is in $D_{\alpha_{k+q}}.$
Consequently, $\widetilde{D}_{\alpha_{k}}+D_{\alpha_{k}}\subseteq D_{\alpha_{k+q}}.$
If one of these inclusions is strict, then $\bigcup_{j}f_{j}\left(C(\alpha)\right)$
would be a strict subset of $C(\alpha),$ by (\ref{Sec-4-Eq:C_alpha defined}).
Hence, (\ref{eq:sp-sparse}) holds. \end{proof}
\begin{example}
Let $D=\{0,2,4,6\},$ $n=7,$ and $\alpha=0.2\overline{0}.$ Then
if follows from Theorem \ref{Sec-4-thm:self-similar} that $C(\alpha)$
is self-similar. However, the self-similarities constructed in the
proof of Theorem \ref{Sec-4-thm:self-similar} do not satisfy the
open set condition. Hence, $C(\alpha)$ is perhaps better understood
in terms of Theorem \ref{Sec-1-Thm:Structure-when-t-is-rational}
where $C(\alpha)$ is described as a finite union of disjoint translates
of a deleted digits Cantor set. \end{example}
\begin{rem}
\label{Sec-4-Remark:-periodin-not-strongly periodic}After we circulated
the first version of this paper, containing a version of Theorem \ref{Sec-4-thm:self-similar}
valid for uniform sets, Derong Kong asked us to provide a set of similaries
for the set $C(\alpha),$ when $D=\{0,2,4,8\},$ $n=9,$ and $\alpha=0.2\overline{0}.$
This is not possible, since $\alpha$ satisfies (\ref{eq:sp-uniform}),
but does not satisfy (\ref{eq:sp-sparse}). We replied to Derong Kong
query that we had a proof of Theorem \ref{Sec-4-thm:self-similar}
as stated above. Subsequently Derong Kong sent us a preliminary version
of the manuscript \cite{Kon12} containing a similar result. Our Theorem
\ref{Sec-4-thm:self-similar} is similar to \cite[Theorem 2.3]{Kon12},
however \cite[Theorem 2.3]{Kon12} shows that $C(\alpha)$ is generated
by similarities $f_{b}(x)=n^{-q}x+b$ from the assumption that $C(\alpha)$
is generated by similarities $f_{b}(x)=rx+b,$ $0<\left|r\right|<1.$ 
\end{rem}

\section{\label{Sec-5:Non-Self-Similar Sets}A Construction of Numbers not
Translation Equivalent to a Rational}

The structure of $C\cap\left(C+\alpha\right)$ is determined by the
previous sections whenever $\alpha$ is translation equivalent to
a rational number. However, there exist many elements $\alpha$ in
$F$ such that $C\left(\alpha\right)$ is not a finite union of self-similar
sets in the sense of Theorem \ref{Sec-1-Thm:Structure-when-t-is-rational}.
Lemma \ref{Sec-4-Lem:Subset_Equivalence} allows us to construct a
family of values $\gamma\in F^{+}$ which are not translation equivalent
to a rational number. In fact, the proof below associates such an
uncountable family of such $\gamma$ to any rational $\alpha$ for
which $C\cap\left(C+\alpha\right)$ is infinite.
\begin{prop}
\label{Sec-5-Pro:alpha not equivalent to a rational}Let $D$ be sparse
and $d_{m}<n$. There exists an uncountably infinite family of values
$\gamma\in F^{+}$ which are not translate equivalent to any rational
number.\end{prop}
\begin{proof}
Let $\alpha$ be a rational such that $C\cap\left(C+\alpha\right)$
is not finite. We may assume $\alpha:=0._{n}\overline{\alpha_{1}\ldots\alpha_{p}}$
by Lemma \ref{Sec-3-Lem:Contained_(t-t_k)}. Fix $i\in\mathbb{N}$
according to the proof of \ref{Sec-4-Lem:Finite C_a} such that $1\le i\le p$
and $\left\{ 0\right\} \varsubsetneqq D_{\alpha_{i}}$ and let $\delta\in\left\{ \delta\in\Delta^{+}\mid D_{\alpha_{i}}\nsubseteq D_{\delta}\right\} $
be arbitrary (this set is nonempty since $D_{d_{m}}=\left\{ 0\right\} $
for any digits set). Suppose $x\in\left[0,1\right]$ has binary representation
$x:=0._{2}x_{1}x_{2}\ldots$ and define $f:\left[0,1\right]\to F$
such that $f\left(x\right)=0._{n}\gamma_{1}\gamma_{2}\ldots$ consists
of digits
\[
\gamma_{j}:=\begin{cases}
x_{h+1}\cdot\alpha_{j}+\left(1-x_{h+1}\right)\cdot\delta & \text{ if }j=i+hp\text{ for some }h\in\mathbb{N}_{0}\\
\alpha_{j} & \text{ otherwise}.
\end{cases}
\]

Thus, $\left\{ 0\right\} \varsubsetneqq D_{\gamma_{i+jp}}=D_{\alpha_{i+jp}}$
if $x_{j+1}=1$ and $D_{\gamma_{i+jp}}=D_{\delta}$ if $x_{j+1}=0$
so that $f\left(x\right)$ is irrational whenever $x$ is irrational.
Since $C\cap\left(C+\alpha\right)$ is infinite then 
\begin{equation}
\left\{ j\mid x_{j+1}=1\right\} =\left\{ j\mid D_{\gamma_{i+jp}}=D_{\alpha_{i+jp}}\right\} \label{Sec-5-Eq:x_j equivalence}
\end{equation}
is an infinite subset of $\mathbb{N}$.

Suppose $\tau:=0._{n}\tau_{1}\tau_{2}\cdots\tau_{h}\overline{\tau_{h+1}\cdots\tau_{h+q}}$
is translate equivalent to $f\left(x\right)$ for some $h\in\mathbb{N}$
and period $q$. Then $D_{\gamma_{h+j}}=D_{\gamma{}_{h+j+q}}$ for
all $j>0$ according to Corollary \ref{Sec-4-Cor:C(a) equals C(g) iff Da_k equals Dg_k}.
If $a$ and $b$ are positive integers satisfying $h+a=bp$, then
for each integer $j>b$,
\[
D_{\gamma_{i+jp}}=D_{\gamma_{h+a+i+\left(j-b\right)p}}=D_{\gamma_{h+a+i+\left(j-b\right)p+pq}}=D_{\gamma_{i+\left(j+q\right)p}}.
\]

Equivalently, $x_{j+1}=x_{j+q+1}$ for all $j>b$ by equation (\ref{Sec-5-Eq:x_j equivalence})
so that $x$ is rational with period $q$.
\end{proof}
If $K\subseteq\mathbb{R}^{n}$ is an arbitrary compact set with $\dim_{H}\left(K\right)$-dimensional
Hausdorff measure $0$ or $\infty$, then $K$ is not a self-similar
set, see e.g., \cite{Fal85} and \cite{Hut81}. In particular, such
a set $K$ cannot be expressed as the finite union of self-similar
sets. In \cite{PePh11b}, a method was given for constructing values
$y\in F$ which satisfy $0<s:=\dim_{H}\left(C\left(y\right)\right)<\log_{n}\left(m\right)$
and $\mathscr{H}^{s}\left(C\left(y\right)\right)=0$ so that such
elements $y$ are not translation equivalent to any rational. Example
\ref{Sec-5-Ex:Not equivalent to a rational.} constructs $\gamma\in F$
which is not translation equivalent to a rational, yet $0<\mathscr{H}^{s}\left(C\left(\gamma\right)\right)<\infty$.
\begin{example}
\label{Sec-5-Ex:Not equivalent to a rational.}Let $D=\left\{ 0,3,6,12\right\} $
and $n=17$. Choose $\alpha:=0._{n}\overline{3}$ so that $C\cap\left(C+\alpha\right)=C_{n,\left\{ 3,6\right\} }$
is self-similar with Hausdorff dimension $s:=\log_{n}\left(2\right)$
and Hausdorff measure 
\[
\mathscr{H}^{s}\left(C\cap\left(C+\alpha\right)\right)=\left(\frac{3}{16}\right)^{s}.
\]
Since $D_{3}=\left\{ 0,3\right\} $ and $D_{6}=\left\{ 0,6\right\} $,
define $\gamma:=0._{n}\gamma_{1}\gamma_{2}\ldots$ such that $\gamma_{j}=3=a_{j}$
except $\gamma_{k}=6$ on a sufficiently sparse set of $k$'s. Thus,
$\gamma$ is irrational and not translate equivalent to any rational
by Proposition \ref{Sec-5-Pro:alpha not equivalent to a rational}
so that $C\cap\left(C+\gamma\right)\cap\left[0,\varepsilon\right]$
is not self-similar for any $\varepsilon>0$. Note, however, that
$\mu_{\alpha}\left(k\right)=\mu_{\gamma}\left(k\right)$ for all $k$.
According to \cite{PePh11b}, $L_{t}=1$ and $\frac{1}{4}\le\mathscr{H}^{s}\left(C\cap\left(C+\gamma\right)\right)\le1$.\end{example}
\begin{rem}
\label{Sec-5-Rem:Uniform D_alpha_j}Suppose $D$ is uniform and $\alpha$
has a $\Delta^{+}$ representation. Since $\Delta^{+}=D$, then $\alpha\in C_{n,D}\subset F$.
Then $D_{\alpha_{j}}=\left\{ 0,d,\ldots,\left(d_{m}-\alpha_{j}\right)\right\} $
for each $\alpha_{j}\in\Delta^{+}$ so that $D_{\alpha_{j}}=D_{\alpha_{i}}$
if and only if $\alpha_{j}=\alpha_{i}$. Hence, an irrational value
$\alpha\in C_{n,D}$ is not translation equivalent to any rational
by Corollary \ref{Sec-4-Cor:C(a) equals C(g) iff Da_k equals Dg_k}
and Lemma \ref{Sec-4-Lem:Subset_Equivalence}.
\end{rem}

\section{\label{Sec-6:Uniform Sets}Uniform Sets}

In this section we consider uniform digits sets and prove Theorem
\ref{Sec-1-Thm:Uniform union of self-similar sets}. This allows us
to prove Theorem \textcolor{red}{\ref{Sec-4-Thm:C_alpha-cap-[0,e] is self-similar}}
with fewer restrictions on the similitudes and to establish connections
to some of the results in the papers mentioned in the introduction. 

The next lemma is a step in that direction. The lemma also allows
us to consider certain $\beta$-expansions with non-uniform digit
sets. 
\begin{lem}
\label{Sec-6-Lem:g_N defined}Let $D$ be sparse, $N:=d_{m}+1$, and
$C=C_{N,D}$. Fix $\beta\in\left(0,\frac{1}{N}\right]$ and suppose
$\alpha:=0._{N}\alpha_{1}\alpha_{2}\ldots$ is a $\Delta^{+}$ representation.
If there exists $\varepsilon>0$ such that $\left(C\cap\left(C+\alpha\right)\right)\cap\left[0,\varepsilon\right]=\left\{ \sum_{k=1}^{\infty}x_{k}\left(\frac{1}{N}\right)^{k}\mid x_{k}\in D\cap\left(D+\alpha_{k}\right)\right\} \cap\left[0,\varepsilon\right]$
is a self-similar set generated by similarities $\left\{ f_{j}\right\} _{j=1}^{\ell}$,
then there exists $\delta>0$ such that 
\[
\left\{ \sum_{k=1}^{\infty}x_{k}\beta^{k}\mid x_{k}\in D\cap\left(D+\alpha_{k}\right)\right\} \cap\left[0,\delta\right]
\]
 is also a self-similar set.\end{lem}
\begin{proof}
Let $D$ be sparse, $N:=d_{m}+1$, and $\alpha\in F^{+}$ be fixed.
The result is trivial if $\beta=N^{-1}$, so suppose $\beta\in\left(0,\frac{1}{N}\right)$
and $\left(C\cap\left(C+\alpha\right)\right)\cap\left[0,\varepsilon\right]=T$
is self-similar for some $\varepsilon>0$. Since $C\cap\left(C+\alpha\right)$
is compact, we may assume $\varepsilon=\sup\left(T\right)\in C$ without
loss of generality. 

Each $\gamma\in C$ has a unique representation $0._{N}\gamma_{1}\gamma_{2}\cdots$
where each $\gamma_{k}\in D$ by equation (\ref{Sec-1-eq:C-defined})
and Lemma \ref{Sec-4-Lem:Unique Delta Representations}. Define $g_{\beta}:C\to\mathbb{R}$
such that 
\[
g_{\beta}\left(\sum_{k=1}^{\infty}\frac{\gamma_{k}}{N^{k}}\right)=\sum_{k=1}^{\infty}\gamma_{k}\beta^{k}.
\]
The function $g_{\beta}$ is both continuous and increasing on $C$,
and $g_{\beta}\left(X\right)\cup g_{\beta}\left(Y\right)=g_{\beta}\left(X\cup Y\right)$
for any $X,Y\subseteq C$. By equation (\ref{Sec-4-Eq:C_alpha defined}),
$g_{\beta}\left(C\cap\left(C+\alpha\right)\right)=\left\{ \sum_{k=1}^{\infty}x_{k}\cdot\beta^{k}\mid x_{k}\in D\cap\left(D+\alpha_{k}\right)\right\} $.
Since any $\varepsilon<\gamma\in C$ implies $g_{\beta}\left(\varepsilon\right)<g_{\beta}\left(\gamma\right)$
then $g_{\beta}\left(T\right)=g_{\beta}\left(C\cap\left(C+\alpha\right)\right)\cap\left[0,\delta\right]$
where $\delta:=g_{\beta}\left(\varepsilon\right)$. 

For arbitrary elements $\gamma\neq\xi$ in $C$, there exists $k\in\mathbb{N}$
such that $\gamma_{k}\neq\xi_{k}$ and $g_{\beta}\left(\gamma\right)\neq g_{\beta}\left(\xi\right)$.
Hence, $g_{\beta}$ has unique inverse for any element of $g_{\beta}\left(C\right)=\left\{ \sum_{k=1}^{\infty}x_{k}\cdot\beta^{k}\mid x_{k}\in D\right\} $
and 
\begin{align*}
\bigcup_{j=1}^{\ell}\left(g_{\beta}\circ f_{j}\circ g_{\beta}^{-1}\right)\left(g_{\beta}\left(T\right)\right) & =\bigcup_{j=1}^{\ell}g_{\beta}\left(f_{j}\left(T\right)\right)\\
 & =g_{\beta}\left(\bigcup_{j=1}^{\ell}f_{j}\left(T\right)\right)\\
 & =g_{\beta}\left(T\right).
\end{align*}

Therefore, $g_{\beta}\left(T\right)$ is a self-similar set generated
by similarities $\left\{ g_{\beta}\circ f_{j}\circ g_{\beta}^{-1}\right\} _{j=1}^{\ell}$.\end{proof}
\begin{cor}
Let $D$ be sparse, $N:=d_{m}+1$, $\beta\in\left(0,\frac{1}{N}\right]$,
and $\alpha$ have $\Delta$ representation. The $\beta$-expansion
Cantor set $g_{\beta}\left(C\cap\left(C+\alpha\right)\right)$ can
be expressed as the disjoint union \textup{$\bigcup_{j=1}^{\ell}\left(g_{\beta}\left(C_{n^{2p},E}\right)+\eta_{j}\right)$
}for some $\eta_{1}<\eta_{2}<\cdots<\eta_{\ell}$ if and only if there
exist integers $k,q$ such that $D_{\left|\alpha_{j}\right|}\subseteq D_{\left|\alpha_{j+q}\right|}$
for all $j>k$.
\end{cor}
The additional structure of uniform sets allows us to prove Theorem
\ref{Sec-1-Thm:Rational_tau} with fewer restrictions on the similitudes.
Theorem \ref{Sec-4-Thm:C_alpha-cap-[0,e] is self-similar} requires
the contraction ratios to be of the form $r_{j}=n^{-q_{j}}$ for some
integers $q_{j}$, however, when $D$ is uniform we require only that
the contraction ratios $r_{j}$ are positive.
\begin{thm}
\label{Sec-6-Thm:Uniform D and r_j>0} Let $D$ be uniform and $\alpha$
have $\Delta^{+}$ representation. Suppose there exists $\varepsilon>0$
such that $C(\alpha)\cap\left[0,\varepsilon\right]$ is a self-similar
set generated by similarities $f_{j}(x)=r_{j}x+b_{j}$ where $r_{j}>0$.
Then $\alpha$ is translation equivalent to a rational number.\end{thm}
\begin{proof}
Since $D$ is uniform, then there exists $d\ge2$ such that $d_{j}=\left(j-1\right)d\in D$
for each $1\le j\le m<n$. Furthermore, $D_{\alpha_{i}}=\left\{ a-\alpha_{i}\mid a\in D\text{ and }a\ge\alpha_{i}\right\} \subseteq D$
by Remark \ref{Sec-4-Rem:Uniform D} so that $C(\alpha)\subseteq C$
and if $\left(j-1\right)d\in D_{\alpha_{i}}$ for some $i$ and $j$
then $\left(k-1\right)d\in D_{\alpha_{i}}$ for all $1\le k\le j$.
We may assume that $b_{1}<b_{2}<\cdots<b_{\ell}$ so that $b_{1}=0$
and $f_{1}\left(x\right)=r_{1}\cdot x$.

We only need show that there exist integers $k,q$ such that $D_{\alpha_{j}}\subseteq D_{\alpha_{j+q}}$
for all $j>k$ by Lemma \ref{Sec-4-Lem:Subset_Equivalence}. According
to Lemma \ref{Sec-6-Lem:g_N defined}, the result holds if there exists
an $n>d_{m}$ such that $D_{\alpha_{j}}\subseteq D_{\alpha_{j+q}}$
for all sufficiently large $j$. Suppose $n>d_{m}\cdot\left(m-1\right)$.

Let $\varepsilon>0$ be a value such that $C(\alpha)\cap\left[0,\varepsilon\right]=T$
is a self-similar set. We may assume by Lemma \ref{Sec-4-Lem:Finite C_a}
that $C(\alpha)$ is not a finite set and $K:=\left\{ k\mid\left\{ 0\right\} \varsubsetneqq D_{\alpha_{k}}\right\} $
is an infinite subset of $\mathbb{N}$. 

Choose $k\in\mathbb{N}$ such that $\varepsilon\ge n^{-k}$ and let
$j>k$ be an arbitrary element of $K$. Then $d\in D_{\alpha_{j}}$
so that $d\cdot n^{-j}\in T\subseteq C(\alpha)$. Thus, $d\cdot r_{1}\cdot n^{-j}\in T\subset C$
and $d\cdot r_{1}\cdot n^{-j}=\sum_{i=1}^{\infty}x_{i}\cdot n^{-i}$
has a unique expression with each $x_{i}\in D$ by definition of $C$
and Lemma \ref{Sec-4-Lem:Unique Delta Representations}. Now, $0<r_{1}<1$
so that $d\cdot r_{1}=\sum_{i=1}^{\infty}x_{i}\cdot n^{j-i}<d$ and
$x_{i}=0$ for all $1\le i\le j$. Since each $x_{i}\in D$ we can
write the $n$-ary representation $r_{1}=0._{n}r_{1,1}r_{1,2}\ldots$
where $r_{1,i}=\frac{x_{i-j}}{d}\in\left\{ 0,1,\ldots,m-1\right\} $
for each $i\in\mathbb{N}$.

Suppose there exists $q\in\mathbb{N}$ such that $r_{1,q}\ge2$. Let
$a_{1}:=\max\left(D_{\alpha_{j}}\right)>0$ so that $a_{1}\cdot n^{-j}\in T$.
We will inductively define a sequence $\left\{ a_{h}\right\} \subseteq D$
so suppose $a_{h}\in D_{\alpha_{j+q\left(h-1\right)}}$ for some $h$.
Then $0<a_{h}\le d_{m}$ and $1<r_{1,i}\le m-1$ so that $a_{h}\cdot r_{1,i}\le d_{m}\cdot\left(m-1\right)$
for all $i$. Hence, $\frac{a_{h}}{n^{j+q\left(h-1\right)}}\cdot r_{1}=\sum_{i=1}^{\infty}a_{h}\cdot r_{1,i}\cdot n^{-i-j-q\left(h-1\right)}$
and $a_{h}<a_{h}\cdot r_{1,q}\in D_{\alpha_{j+qh}}$. Define $a_{h+1}:=\max\left(D_{\alpha_{j+q\cdot h}}\right)>a_{h}$. 

Thus, we have defined $a_{1}<a_{2}<\ldots<a_{m}$ such that $a_{h}=\max\left(D_{\alpha_{j+q\cdot h}}\right)\in D$
for all $1\le h\le m$. Since $D$ is a uniform set containing $m$
elements then $D=\left\{ a_{1},a_{2},\ldots,a_{m}\right\} $. This
leads to a contradiction since $a_{1}>0$ yet $0\in D$.

Therefore, $r_{1,i}\in\left\{ 0,1\right\} $ for all $i\in\mathbb{N}$
and there exists $q$ such that $r_{1,q}=1$ since $r_{1}>0$. If
$d_{i}\in D_{\alpha_{j}}$ for some $1\le i\le m$, then $r_{1}\cdot d_{i}\cdot n^{-j}\in C_{\alpha}$
and $d_{i}\in D_{\alpha_{j+q}}$. Hence, $D_{\alpha_{j}}\subseteq D_{\alpha_{j+q}}$
for all $j>k$.
\end{proof}
It is a simple consequence of the proof of Theorem \ref{Sec-6-Thm:Uniform D and r_j>0}
that each element $q\in\left\{ i\mid r_{1,i}=1\right\} $ is a period
of the rational number that is translation equivalent to $\alpha$.
It is also interesting to note that the proof of Theorem \ref{Sec-3-Thm:Rational_tau}
constructs a collection $\left\{ f_{j}\left(x\right)=r_{j}x+b_{j}\right\} _{j=1}^{\ell}$
where each $r_{1}=r_{2}=\cdots=r_{\ell}=n^{-q}$ for some $q\in\mathbb{N}$,
however the set $\left\{ i\mid r_{1,i}=1\right\} $ could be countably
infinite in the proof of Theorem \ref{Sec-6-Thm:Uniform D and r_j>0}.
\begin{rem}
\label{Sec-4-Rem:Uniform D}Suppose $D$ is a uniform digit set. Then
$D_{\alpha_{j}}=\left\{ a-\alpha_{j}\mid a\in D\text{ and }a\ge\alpha_{j}\right\} $
for each $\alpha_{j}\in\Delta^{+}$ and $D_{\alpha_{j}}\subseteq D$
for all $j$ by definition of uniform sets so that $C\left(\alpha\right)\subseteq C$.
This also implies that $D_{\alpha_{j}}=\left\{ 0,d,\ldots,d\left(d_{m}-\alpha_{j}\right)\right\} =\max\left\{ D_{\alpha_{j}}\right\} -D_{\alpha_{j}}$
for all $j$ so that $C(\alpha)=z-C(\alpha)$ is centrally symmetric
when $z:=\sum_{j=1}^{\infty}\left(\max\left\{ D_{\alpha_{j}}\right\} \cdot n^{-j}\right)$. 

When $D$ is regular, but not uniform, then $C\left(\alpha\right)$
need not be a subset of $C$. For example, we can choose $D=\left\{ 0,4,6,8\right\} $,
$n\ge9$, and $\alpha:=\frac{2}{n}$. Thus, $D_{\alpha_{1}}=\left\{ 0,2\right\} $
so that $\alpha\in C(\alpha)$ yet $\alpha\notin C$.
\end{rem}
According to Remark \ref{Sec-4-Rem:Uniform D}, if $D$ is uniform
and $C\left(\alpha\right)$ is a self-similar set then we may choose
$\varepsilon=1$ to obtain the following Corollary to Theorem \ref{Sec-6-Thm:Uniform D and r_j>0}:
\begin{cor}
\label{Sec-6-Cor:Uniform with epsilon=00003D1}Let $D$ be uniform
and $\alpha$ have $\Delta^{+}$ representation. If $C(\alpha)$ is
a self-similar set generated by similarities $f_{j}(x)=r_{j}x+b_{j}$
then $\alpha$ is translation equivalent to a rational number.\end{cor}
\begin{proof}
It is sufficient to show that $T:=C\left(\alpha\right)$ can be generated
by a collection of similarity mappings with positive contraction ratios.
Since $T$ is centrally symmetric, then $T=z-T$ for some $z\in\left[-1,1\right]$.
For each $1\le j\le\ell$, let $h_{j}\left(x\right):=-r_{j}x+\left(b_{j}+z\cdot r_{j}\right)$
if $r_{j}<0$, otherwise let $h_{j}\left(x\right):=f_{j}\left(x\right)$.
Then
\[
\bigcup_{j=1}^{\ell}h_{j}\left(T\right)=\left(\bigcup_{r_{j}<0}f_{j}\left(z-T\right)\right)\cup\left(\bigcup_{r_{j}>0}f_{j}\left(T\right)\right)=\bigcup_{j=1}^{\ell}f_{j}\left(T\right)=T.
\]

\end{proof}
This completes the proof of Theorem \ref{Sec-1-Thm:Uniform union of self-similar sets}.
A special case of Corollary \ref{Sec-6-Cor:Uniform with epsilon=00003D1}
is proven in \cite{LYZ11} when $d_{m}=n-1$.

\subsection{\label{Sub-6.1:Uniform and strongly periodic rationals}Uniform sets
and strongly periodic rationals}

In this section we assume that $D$ is a uniform digits set. Recall,
when $D$ is uniform, a sequence $\left\{ \alpha_{k}\right\} \subseteq\Delta^{+}$
is strongly periodic if and only if there exists an integer $q>0$
such that $D_{\alpha_{j}}\subseteq D_{\alpha_{j+q}}$ for all $j>0$.
In this section we show that this is consistent with the definition
given in \cite{DHW08} and \cite{LYZ11}.
\begin{prop}
Let $D$ be uniform, $n=d_{m}+1$, and $\alpha=0._{n}\alpha_{1}\alpha_{2}\ldots$
have $\Delta$ representation. Define $\hat{\alpha}:=0._{n}\hat{\alpha}_{1}\hat{\alpha}_{2}\cdots$
such that $\hat{\alpha}_{k}:=d_{m}-\left|\alpha_{k}\right|$ for each
$k$. There exists an integer $q>0$ such that $D_{\left|\alpha_{j}\right|}\subseteq D_{\left|\alpha_{j+q}\right|}$
for all $j>0$ if and only if there exist $\boldsymbol{u},\boldsymbol{v}\in D^{p}$
for some integer $p>0$ such that $u_{j}\le v_{j}$ for all $1\le j\le p$
and $\hat{\alpha}=0._{n}u_{1}\cdots u_{p}\overline{v_{1}\cdots v_{p}}$.\end{prop}
\begin{proof}
Suppose $\boldsymbol{u},\boldsymbol{v}\in D^{p}$ such that $u_{j}\le v_{j}$
for all $1\le j\le p$ and $\hat{\alpha}=0._{n}u_{1}\cdots u_{p}\overline{v_{1}\cdots v_{p}}$.
Since $u_{j},v_{j}\in D$ then the $\Delta^{+}$ representation of
$\hat{\alpha}$ is unique, so that $u_{j}\le v_{j}$ for all $1\le j\le p$
is equivalent to $d_{m}-\left|\alpha_{j}\right|\le d_{m}-\left|\alpha_{j+p}\right|$
for all $j>0$ by definition of $\hat{\alpha}_{j}$. Furthermore,
$d_{m}-\left|\alpha_{j}\right|\le d_{m}-\left|\alpha_{j+p}\right|$
if and only if $D_{\left|\alpha_{j}\right|}\subseteq D_{\left|\alpha_{j+p}\right|}$
since $D_{\left|\alpha_{j}\right|}=\left\{ 0,d,\ldots,\left(d_{m}-\left|\alpha_{j}\right|\right)\right\} $
for all $\alpha_{j}\in\Delta$.
\end{proof}
Thus, when $D$ is uniform, we extend the definition of strongly periodic
to mean there exists $q>0$ such that $D_{\left|\alpha_{j}\right|}\subseteq D_{\left|\alpha_{j+q}\right|}$
for all $j>0$. We note that $\alpha$ need not be rational to satisfy
this equation, but any such $\alpha$ is translation equivalent to
a rational. We point out that if $D$ is uniform, then $F$ contains
three disjoint partitions:
\begin{enumerate}
\item According to Theorem \ref{Sec-1-Thm:Uniform union of self-similar sets},
if $\alpha\in F$ is translation equivalent to a strongly periodic
rational then $C\left(\alpha\right)\cap\left[0,\varepsilon\right]$
is a self-similar set for $\varepsilon=1$.
\item If $\alpha\in F$ is translation equivalent to a rational $\gamma$,
but not to any strongly periodic rational, then $C\left(\alpha\right)\cap\left[0,\varepsilon\right]$
is a self-similar set for some $0<\varepsilon<1$.
\item Otherwise, if $\alpha\in F$ is not translation equivalent to any
rational, then $C\left(\alpha\right)\cap\left[0,\varepsilon\right]$
is not a self-similar set for any $\varepsilon>0$.
\end{enumerate}
Example \ref{Sec-6-Ex:MTC is self-similar} illustrates a case when
$\alpha$ is a strongly periodic rational.
\begin{example}
\label{Sec-6-Ex:MTC is self-similar}Let $C=C_{3,\left\{ 0,2\right\} }$
denote the middle thirds Cantor set and $\alpha:=0._{3}02\overline{0}$.
Then $C\cap\left(C+\alpha\right)$ consists of $\mu_{\alpha}\left(2\right)=2$
disjoint copies of $\frac{1}{9}C$ by Theorem \ref{Sec-3-Thm:Rational_tau}.
Let $q=2$ so that $D_{\alpha_{j}}\subseteq D_{\alpha_{j+q}}$ for
all $j>0$ and $\alpha$ is strongly periodic. Thus, $C\cap\left(C+\alpha\right)$
is a self-similar set composed of two ``smaller'' copies of $C$.
Furthermore, the Hausdorff dimension of $C\cap\left(C+\alpha\right)$
is $s:=\log_{3}\left(2\right)$ and the Hausdorff measure is $\mathscr{H}^{s}\left(C\cap\left(C+\alpha\right)\right)=\frac{1}{2}$.
\end{example}
Example \ref{Sec-6-Ex:MTC is not self-similar} demonstrates a rational
in $F$ that is not strongly periodic.
\begin{example}
\label{Sec-6-Ex:MTC is not self-similar}Let $C=C_{3,\left\{ 0,2\right\} }$
denote the middle thirds Cantor set and $\alpha:=0._{3}02\overline{20}$.
Then $C\cap\left(C+\alpha\right)$ consists of $\mu_{\alpha}\left(2\right)=2$
disjoint copies of $\frac{1}{9}C_{9,\left\{ 6,8\right\} }$ by Theorem
\ref{Sec-3-Thm:Rational_tau}. If $s:=\log_{9}\left(2\right)$ then
$\mathscr{H}^{s}\left(C\cap\left(C+\alpha\right)\right)=4^{-s}$ according
to Proposition \ref{Sec-3-Prop:Hausdorff Measure when D=00003D{0,d}}.
However, if $q=2k$ is even then $D_{\alpha_{1}}=\left\{ 0,2\right\} $
and $D_{\alpha_{1+2k}}=\left\{ 0\right\} $. Similarly, if $q=2k+1$
then $D_{\alpha_{4}}=\left\{ 0,2\right\} $ and $D_{\alpha_{4+2k+1}}=\left\{ 0\right\} $.
Hence, $\alpha$ is not translation equivalent to any strongly periodic
rational and $C\cap\left(C+\alpha\right)$ is not a self-similar set.
\end{example}

\subsection{\label{Sec-6.2:beta expansions}$\beta$-expansion Cantor Sets}

Let $N\ge2$, $\Omega\subseteq\left\{ 0,1,\ldots,N-1\right\} $ be
an arbitrary set containing at least two elements, and $\beta\in\left(0,\frac{1}{N}\right)$.
If $\phi_{d}\left(x\right):=\beta x+d\left(1-\beta\right)/\left(N-1\right)$,
then the set generated by $\left\{ \phi_{d}\mid d\in\Omega\right\} $
is the \emph{$\beta$-expansion Cantor set} 
\[
\Gamma_{\beta,\Omega}:=\left\{ \sum_{k=1}^{\infty}\frac{x_{k}\beta^{k-1}\left(1-\beta\right)}{\left(N-1\right)}\mid x_{k}\in\Omega\right\} .
\]

According to Lemma \ref{Sec-6-Lem:g_N defined}, if there exists an
integer $d\ge1$ such that $D=d\cdot\Omega$ is a sparse digits set
and $d_{m}\le N-1$, then $\Gamma_{\beta,D}$ can be expressed as
\[
\Gamma_{\beta,D}=\frac{\left(1-\beta\right)}{\beta\left(N-1\right)}\cdot g_{\beta}\left(C_{N,D}\right)
\]
for some sparse deleted digits Cantor set $C_{N,D}$. Therefore, when
$\beta$ is small it is sufficient to consider the structure of deleted
digits Cantor sets. We point out that $g_{\beta}$ only preserves
the structure of these sets; the Hausdorff dimension and measure are
not necessarily preserved since $g_{\beta}\left(C_{3,\left\{ 0,2\right\} }\right)$
has dimension $\log_{\frac{1}{\beta}}\left(2\right)$ for any $\beta\in\left(0,\frac{1}{3}\right)$.
Our results do not necessarily hold for larger values of $\beta$
since $\Gamma_{\beta,\Omega}-\Gamma_{\beta,\Omega}$ may not satisfy
the open set condition. We refer to \cite{ZLL08} and \cite{KLD12}
for analysis of uniform $\beta$-expansion Cantor sets when $\beta>\frac{1}{d_{m}+1}$.

Many of our results support the idea that self-similarity structure
is determined by the sequence $\left\{ \alpha_{k}\right\} \mbox{\ensuremath{\subseteq\Omega}}$,
sometimes called the \emph{$\Omega$-code}. If $D$ is sparse and
$d_{m}<n$, then $F$ satisfies the open set condition and any $\Delta^{+}$
representations are unique by Lemma \ref{Sec-4-Lem:Unique Delta Representations}.
We avoid (direct) discussion of $\Omega$-codes to focus on the geometry
of $n$-ary intervals $J^{\left(h\right)}\subset C_{k}$. Lemma \ref{Sec-6-Lem:g_N defined}
directly supports the idea that self-similarity is independent of
the chosen base when $\beta$ is small.

\section*{Acknowledgement}

The co-authors thank Derong Kong for making them aware of the example
in Remark \ref{Sec-4-Remark:-periodin-not-strongly periodic}.

\end{document}